\documentclass[english]{amsart}

\usepackage{amsmath, amsfonts, amssymb, mathrsfs, mathabx, amsthm}
\usepackage[shortlabels]{enumitem}
\usepackage[english]{babel}
\usepackage[utf8]{inputenc}
\usepackage[T1]{fontenc}		
\usepackage[text={445pt,630pt}, centering, a4paper,
marginparwidth=2cm]{geometry}
\usepackage{tikz}
\usepackage{graphicx}
\usepackage{tabularx}
\usepackage[hidelinks,pdfusetitle]{hyperref}
\usepackage{bbm}
\usepackage{subcaption}
\usepackage{cancel}
\usepackage{bbold}

\DeclareMathOperator{\card}{card}

\DeclareMathOperator{\id}{Id}

\DeclareMathOperator{\R}{\mathbb{R}}
\DeclareMathOperator{\C}{\mathbb{C}}

\DeclareMathOperator{\N}{\mathbb{N}}

\DeclareMathOperator{\Reel}{\mathcal{R}e}

\DeclareMathOperator*{\essinf}{ess inf}
\DeclareMathOperator*{\esssup}{ess sup}
\DeclareMathOperator*{\limessinf}{lim ess inf}
\DeclareMathOperator*{\limesssup}{lim ess sup}
\DeclareMathOperator*{\limess}{lim ess}
\DeclareMathOperator{\supp}{supp}

\DeclareMathOperator{\spec}{Sp}

\newcommand{\atom}{\mathfrak{A}}
\newcommand{\atomsc}{\mathfrak{A}^{\mathrm{sup}}}

\newcommand{\anz}{\mathfrak{A}^*}

\newcommand{\cc}{\mathcal{C}}

\newcommand{\cf}{\mathcal{F}}

\newcommand{\cg}{\mathcal{G}}

\newcommand{\cv}{\mathcal{V}}

\newcommand{\norm}[1]{\left\lVert\,#1\,\right\rVert}
\newcommand{\ind}[1]{\mathbb{1}_{#1}}
\newcommand{\expp}[1]{\mathrm{e}^{#1}}
\newcommand{\rd}{\mathrm{d}}

\newcommand{\un}{\mathbb{1}}
\newcommand{\zero}{\mathbb{0}}

\newcommand{\priv}[1]{\backslash\{#1\}}

\newcommand{\eq}{\mathcal{E}}

\newcommand{\TT}{\mathrm{T}}
\newcommand{\ti}{T'}
\newcommand{\tgi}{T_{1/\gamma}}
\newcommand{\tdi}{T_{1/\delta}}
\renewcommand{\tt}{\tilde {T}}
\newcommand{\tp} {\hat{T}}
\newcommand{\tg}[1] {\tp_{1/#1}}
\newcommand{\tgp} {\tp_{1/\gamma}}

\renewcommand{\r}{\mathbf{r}}

\newcommand{\sac}{\mathfrak{S}}

\usepackage[draft]{fixme} 
\FXRegisterAuthor{pa}{apa}{PAZ}
\FXRegisterAuthor{jf}{ajf}{JFD}
\FXRegisterAuthor{kl}{akl}{KL}
\fxusetheme{colorsig}

\newcommand{\scal}[1]{\left\langle #1\right\rangle}

\begin{document}
\title{Infinite dimensional metapopulation SIS model with generalized incidence rate}

\author{Jean-Fran\c{c}ois Delmas}
\address{Jean-Fran\c{c}ois Delmas,  CERMICS, \'{E}cole des Ponts, France}
\email{jean-francois.delmas@enpc.fr}

\author{Kacem Lefki}
\address{Kacem Lefki, Univ.\ Gustave Eiffel, Univ.\
    Paris Est Creteil, CNRS, F-77454 Marne-la-Vall\'ee, France}
  \email{kacem.lefki@univ-eiffel.fr}

\author{Pierre-Andr\'{e} Zitt}
\address{Pierre-Andr\'{e} Zitt, Univ.\ Gustave Eiffel, Univ.\
    Paris Est Creteil, CNRS, F-77454 Marne-la-Vall\'ee, France}
  \email{Pierre-Andre.Zitt@univ-eiffel.fr}

\date{\today}

\thanks{This work is partially supported by Labex B\'ezout reference ANR-10-LABX-58}


, 

\begin{abstract}
  We  consider an  infinite-dimension  SIS model  introduced by  Delmas,
  Dronnier and  Zitt, with a more  general incidence rate, and study its
  equilibria. Unsurprisingly, there exists at
  least one  endemic equilibrium if  and only if the  basic reproduction
  number is larger than 1.  When the pathogen transmission exhibits one
  way propagation, it is possible  to observe different possible endemic
  equilibria.  We characterize  in a general setting  all the equilibria,
  using a decomposition of the space into  atoms, given by the transmission  operator.
  We also
  prove that  the proportion of  infected individuals  converges
  to  an equilibrium,  which is  uniquely determined  by the
  support of the initial condition.

  We  extend those results to 
  infinite-dimensional SIS models with reservoir or  with immigration. 
\end{abstract}

\maketitle

\theoremstyle{theorem}
\newtheorem{theo}{Theorem}[section]
\newtheorem{theoI}{Theorem}
\newtheorem{propI}[theoI]{Proposition}
\newtheorem{prop}[theo]{Proposition}
\newtheorem{defi}[theo]{Definition}
\newtheorem{lemme}[theo]{Lemma}
\newtheorem{lem}[theo]{Lemma}
\newtheorem{coro}[theo]{Corollary}
\newtheorem{cor}[theo]{Corollary}
\newtheorem{nota}[theo]{Notation}
\newtheorem{rap}[theo]{Reminder}
\newtheorem{conj}[theo]{Conjecture}
\newtheorem{ques}[theo]{Question}
\newtheorem{ans}[theo]{Answer}
\newtheorem{defth}[theo]{Theorem-Definition}
\newtheorem{assum}{Assumption}

\theoremstyle{remark} 
\newtheorem{ex}[theo]{Example}
\newtheorem{rqe}[theo]{Remark}
\newtheorem{rem}[theo]{Remark}
\newtheorem{cex}[theo]{Counter-example}
\maketitle

\section{Introduction}

\subsection{Model and relations with existing models}\label{subsec:intro:model}

We consider  an inhomogeneous SIS epidemic model, where
individuals are either susceptible or infected. The homogeneous
model was introduced  by Kermack and McKendrick~\cite{kermack32}, we
refer    to    the    monograph   of    Brauer,    Castillo-Chavez    et
Feng~\cite{brauer19}  for an  analysis of this homogeneous  SIS model  and some  of its
variants. Let us  recast  the  model from~\cite{kermack32}   in the  constant
population case: let
$I(t)$ and   $S(t)$ denote respectively the  number of the infected  and  susceptible)
individuals  at  time $t \geq  0$, in a population  of constant size  $N=S(t)+I(t) >
0$. The evolution of the number of infected is given by:
\begin{equation}\label{eq:SIS_1D}
I' = k \,\frac{SI}{N} - \gamma I,
\end{equation}
where  $k \geq 0$ is   the infection  rate and  $\gamma  > 0$  the
recovery rate.

\medskip

The assumption of homogeneity of  the population is not always satisfied
in practice, see for example: Trauer et al.~\cite{trauer19} for a review
on  tuberculosis, \cite{popkin20}  on  the impact  of health  condition,
\cite{britton07}  on  the  number  of  sexual  partners  in  a  sexually
transmissible  infection,  or  the review~\cite{vanderwaal16}  for  more
possible sources  of heterogeneity.  The inhomogeneous  SIS model from
Lajmanovich     and    Yorke~\cite{lajmanovich76}     generalizes    the
Kermack-McKendrick  model to  a  population divided  in $n$  sub-groups;
the same equation appears also when studying network of communities
linked by dispersal, see Mouquet and Loreau~\cite{mouquet02} and more
generally~\cite{cantrell17}.
Later,     Thieme~\cite{thieme11}    and     Delmas,    Dronnier     and
Zitt~\cite{ddz-sis}  introduced   a  variant  allowing  an infinite
number (possibly uncountable)   of
sub-groups  or  features.

\medskip

We  follow  the model  given  by~\cite{ddz-sis}  where the  transmission
operator    can   be    non-irreducible,   see    the   discussion    in
Section~\ref{intro:sec:unique_ee} below and  allowing furthermore a more
general  incidence   rate,  see   Section~\ref{sec:intro_KMK_LMA}.   The
heterogeneity   of    the   population    is   described    as   follow:
$(\Omega, \cg, \mu)$ is a measured space with a non-zero $\sigma$-finite
measure  $\mu$: an  element $x\in  \Omega$ corresponds  to a  particular
\emph{feature}  (or  \emph{trait})  of   individuals.   We  assume  that
individuals with  the same feature behave  in the same way with  respect to
the  epidemic, and that  features stay 
constant during the whole infection process.   We also assume that for a
given feature $x \in \Omega$, the  size of the population $\mu(\rd x)$ of feature $x$
remains constant over time. 

\medskip

Let $u(t,x)$ denote the proportion of individuals with feature $x\in \Omega$ that are
infected at time $t\geq 0$ among the population  of individuals with
feature  $x$.  Let $\Delta$ be the set of measurable functions defined
on $\Omega$ taking values in $[0, 1]$. The heterogeneous SIS dynamics
is given, for an initial condition $h \in \Delta$, by the evolution
equation on  the  Banach  space
 $L^\infty  $  of measurable  bounded  real-valued  functions defined  on
 $\Omega$ by:
\begin{align}
  \label{eq:intro_gen_SIS}
\left\{\begin{array}{l}
u' = F(u),
\\ 
u(0) = h\in \Delta,
\end{array}\right.
\end{align}
with
\begin{equation}
   \label{eq:def-F}
  F(u)= \varphi(u)\, Tu - \gamma u,
 \end{equation}
 where $F$ depends on: a bounded linear \emph{transmission operator} $T$
 on  $L^\infty $,  a bounded  real-valued positive  \emph{recovery rate}
 function  $\gamma$  defined on  $\Omega$,  and  a real-valued  function
 $\varphi$   defined  on   $\R$   encoding  the   non-bilinearity  of   the
 \emph{incidence    rate}.    The    hypotheses    on   the    parameter
 $(T,   \gamma,   \varphi)$   are  given   in   Assumptions~\ref{assum:1}
 and~\ref{assum:2}. Let us stress that the usual \emph{law of mass
   action}  $\varphi = 1-\id$, with $\id $ the identity map on $\R$,
 satisfies the corresponding hypothesis from  Assumption~\ref{assum:2}
 summarized in Condition~\eqref{eq:hyp-phi}. 

 \begin{rem}[The kernel model from \cite{ddz-sis}]
   \label{rem:sis-ddz}
 Let $k : \Omega^2
\rightarrow \R_+$ be a kernel, that is a nonnegative measurable
function. The associated kernel operator $T_k$ is defined as follow. For
$h\in L^\infty $ and  $x \in \Omega$,  we define:
\[
  T_k (h) (x) = \int_\Omega k(x,y) h(x) \,  \mu( \rd y).
\]
The quantity $k(x,y)$ represents  the transmission rate from individuals
with feature  $y\in \Omega$  to those with  feature $x\in  \Omega$.  The
heterogeneous    SIS   model    from~\cite{ddz-sis}   is    then   given
by~\eqref{eq:intro_gen_SIS}  and~\eqref{eq:def-F}  with  $T=T_k$  (under
some integral hypothesis on the kernel) and the usual law of mass action
$\varphi =  1-\id$.  This  in particular  encompasses the  Lajmanovich and
Yorke model.
 \end{rem}

 In     epidemiology,     \emph{equilibria}      are     constant     solutions
 of~\eqref{eq:intro_gen_SIS}, that is, functions $g\in \Delta$ such that:
\begin{equation}
  \label{eq:gen_equilibre}
F(g) = 0. 
\end{equation}
They play a  significant role in the long-time behavior  of the dynamics
of an  outbreak, see Theorem~\ref{thI:cv_eq_max} below.   Obviously, the
\emph{disease-free equilibrium}  (DFE) $g  = 0$  is an  equilibrium. Any
other  equilibrium  is  called  \emph{endemic  equilibrium}  (EE).   The
\emph{basic reproduction number} denoted $R_0$ is defined by Heesterbeek
and  Dietz~\cite{heesterbeek96} as  ``the expected  number of  secondary
cases  produced  by a  typical  infected  individual during  its  entire
infectious period,  in a  population consisting of  susceptibles only''.
Following~\cite{ddz-sis}  (see also  the method  of the  next-generation
operator in Diekmann, Heesterbeek and Metz~\cite{diekmann90}), the basic
reproduction  number $R_0$  for  the SIS  model~\eqref{eq:intro_gen_SIS}
with the usual incidence rate associated to $\varphi=1- \id$ is
defined as the  spectral radius of the operator  $T M_{1/\gamma}$, where
the operator $M_{1/\gamma}$ is the multiplication by $1/\gamma$.
There usually is a threshold behavior  for the existence of EE according
to  the value  of $R_0$:  for $R_0  \leq 1$  only the  DFE exists  as an
equilibrium, and  for $R_0 > 1$  there exists an EE.
This is not universal: for example, models with imperfect  vaccines
or exogenous  re-infections might
lead to backward  bifurcation and produce multiple  EE even  in the regime
$R_0 < 1$, see~\cite{gumel12}  and more specifically~\cite{brauer04} for
a SIS model. Nevertheless, we check that threshold behavior holds
    for    the     SIS    model~\eqref{eq:intro_gen_SIS},    see
Theorem~\ref{thI:bij_antichains_eq}   below.    A  discussion   of   the
uniqueness of EE is given in Section~\ref{intro:sec:unique_ee}.

\subsection{A taste of the  main results in the finite setting}
\label{intro:sec:unique_ee}

Except in the trivial case where the population may be split in subpopulations
that do not interact at all, the existence of multiple equilibria is fundamentally
linked to \emph{asymmetry} in the transmission dynamics. In this section,
we first explain this
phenomenon, and a related crucial decomposition of the space, in the simple
case where $\Omega$ is finite, to give a taste of the general results stated below.

We consider a finite set   $\Omega$, let  $\cg$  be the set of  subsets of $\Omega$,
and $\mu$ a  finite measure
with support $\Omega$.   The transmission  operator $T$ is
identified  with a  matrix  $K=\left(  K_{x,y} \right)_{x,y\in  \Omega}$
where $K_{x,y}$ is the infection  rate from individuals with feature $y$
to  those with  feature $x$;  in particular  it takes  into account  the
relative size of the sub-populations.   When $\Omega$ is a singleton and
$\varphi = 1-\id$, we recover Equation~\eqref{eq:SIS_1D} (with $u = I/N$
and $K=k$).   When $\Omega$  is finite, we  recover the  Lajmanovich and
Yorke~\cite{lajmanovich76} model, and the same  framework can be used to
describe households models~\cite{ball99} and multi-host and vector-borne
diseases~\cite{mccormack07}.

\medskip

In this finite case, the  study of the non-uniqueness for equilibria
relies   on  the   properties   of  the   oriented  transmission   graph
$G_K=(\Omega,      E_K)$      with       the      set      of      edges
$E_K=\{(y,x)\in \Omega^2\, \colon\, K_{x,y}>0\}$ given by the support of
the  transmission matrix  $K$.   An  edge from  $y$  to  $x$ models  the
possibility of infection from the sub-population with feature $y$ to the
sub-population  with  feature $x$;  in  particular  the graph  may  have
self-loops.    For   transmission   graph   models   see   for   example
\cite{guo08,arino09}.

Strongly connected components of $G_K$ will be called \emph{atoms} ---
the notion will be generalized in the infinite case. 
An atom is \emph{non-zero}  unless it is a singleton with no self
loop.   Notice the  transmission matrix/operator  is irreducible  if and
only if the  graph $G_K$ is strongly connected (that  is, $\Omega$ is an
atom), and  is said monatomic  if there is  a unique non-zero  atom. For
further  result on  monatomic  operators in the general case see  \cite{dlz} and  references
therein;   we  refer   also   to   Corollary~\ref{cor:monoatom}  for   a
characterization of  monatomic transmission  matrix using the  number of
EE.

\medskip

In many examples  the transmission graph $G_K$ is  symmetric (even though
the  transmission  $K$  might  not  be symmetric).  In  this  case,  all
(strongly) connected  components behave independently and  one can study
each connected components separately.
Cases where~$G_K$  is not strongly connected occur
less frequently in  the literature; it has been mentioned  for example in a
multi-type SIR model by~\cite{knipl14, magal18}.

Let us mention two examples of non symmetric transmission graphs.

\begin{enumerate}[(i)]
\item\label{it:nvm} The  West Nile Virus,  presented in~\cite{bowman05},
  infects three species,  birds (B), humans (H) and  mosquitoes (M).  It
  is a vector-borne disease where  birds and mosquitoes serve as vectors
  for a transmission to humans.  In this model, mosquitoes infects birds
  and  humans while  biting them  and mosquitoes  get infected  by birds
  while biting them, and we assume  there is no infection from humans to
  mosquitoes, nor  between birds and  humans.  The graph $G_K$  given in
  Fig.~\ref{fig:WNV_graph} has  only one non  zero atom $\{B,M\}$  and a
  zero atom $\{H\}$. In particular $K$ is monatomic. 
\item  \label{it:zm}  In  the   zoonosis  model  from~\cite{royce20},  a
  pathogen  exists  in wild  animals  (W),  is transmitted  to  domestic
  animals (D) that  transmit themselves the pathogen to  humans (H). The
  graph $G_K$ given in Fig.~\ref{fig:RF_graph} has three non-zero atoms:
  $\{W\}, \{D\}$ and $\{H\}$.
\end{enumerate}

In such cases where  $G_K$ is not symmetric,  the picture is
richer: many endemic equilibria may exist, they may be entirely characterized
by the atoms contained in their support, and their basins of attraction
may be described explicitly.

\medskip

Let us give a few additional definitions to state these
results more precisely, before giving the general statements below in
Theorem~\ref{thI:bij_antichains_eq} and~\ref{thI:cv_eq_max}.  Define
the \emph{future} of a set $A\subset \Omega$ as the set $\cf(A)$ of
all the vertices in $G_K$ reachable from $A$ by a (possibly empty)
path using edges in $E_K$.  For two atoms $A$ and $B$ of $G_K$, we
write $B \preccurlyeq A$ if $ B\subset \cf(A)$; the relation
$\preccurlyeq$ is a partial order.  An \emph{antichain} of atoms is a
set of atoms which are pairwise unordered.  The future of an antichain
is the future of the union of its elements.

In   the  West   Nile  Virus
model~\ref{it:nvm}, the antichains of non-zero atoms are $\emptyset$ and
$\{B,M\}$; in the zoonosis model~\ref{it:zm}, the antichains of non-zero
atoms are: $\emptyset$, $\{H\}$, $\{D\}$  and $\{W\}$.

Finally, an atom is
\emph{supercritical} if the  basic reproduction of the  SIS-model restricted to
the  atom is  larger  than  1; in  particular  a  supercritical atom  is
non-zero   and    an   atom   $A$   is    trivially   supercritical   if
$K_{x,x}/\gamma(x)>1$ for all $x\in A$, where $\gamma$ is the recovery rate
function,  and  the function   $\varphi$  satisfies the  regularity
Condition~\eqref{eq:hyp-phi}     below.

\medskip

Our first main result, Theorem~\ref{thI:bij_antichains_eq},
states (in the general possibly infinite setting) that
each equilibrium  is characterized
by a (different) antichain of supercritical atoms, and the DFE is associated to the
antichain $\emptyset$.   For example,  assuming for simplicity  that all
non-zero atoms are supercritical, we deduce  that in the West Nile Virus
model~\ref{it:nvm}  there  is only  one  EE  and  that in  the  zoonosis
model~\ref{it:zm} there are  three EE.

Let us mention that a similar result  on the existence of multiple EE is
obtained  in Waters  et al.~\cite{waters16}  for a  waterborne parasites
that infect both humans and animals, such as \emph{Giardia} infection in
rural Australia.  In  this model the animals and the  humans can be seen
as  non-zero   atoms  for  the   transmission,  and  the  water   as  an
environmental  reservoirs.    This  model  does  not   fit  exactly  the
metapopulation SIS  model~\eqref{eq:intro_gen_SIS}-\eqref{eq:def-F}, nor
the    model   with    external   disease    reservoir   presented    in
Section~\ref{sec:intro-SISk} because the reservoir is between the animal
population and the human population.

Theorem~\ref{thI:bij_antichains_eq} also states  also that the support  of an
equilibrium is  given by  the future of  its corresponding  antichain of
supercritical atoms.   Furthermore, Corollary~\ref{cor:supp-equi=} asserts
that for two equilibria $g$ and $g'$,  we have $g\leq g'$ if and only if
$\supp(g)\subset \supp(g')$.  This in  particular allows to  recover the
existence of a maximal equilibrium $g^*$, in the sense that if $g$ is an
other equilibrium, then $g\leq g^*$.

For example,  assuming again for simplicity  that all
non-zero atoms are supercritical, we deduce  that  in  the  zoonosis
model~\ref{it:zm}, denoting by $g_x$ the equilibrium characterized by
the antichain $\{x\}$, we have:
\[
  \supp (g_H)=\{H\}\subset \supp (g_D)=\{D, H\}\subset \supp(g_W)=\{W,D,H\}
\quad\text{and}\quad
0\neq g_H\lneqq  g_D\lneqq  g_W=g^*.
\]

\medskip

Our second main result is a full characterization of basins
of attraction of the various equilibria: we show in 
Theorem~\ref{thI:cv_eq_max} that,   starting   with  an   initial   condition
$u(0)=h$, the epidemics  converges in long time  towards the equilibrium
associated to the maximal antichain in $\cf(\supp(h))$ the future of the
support of the initial condition.

In the example  of  the West Nile  Virus model~\ref{it:nvm}, assuming
that all non-zero atoms are  supercritical, we deduce that starting with
an initial  condition where only  the human population $H$  is infected,
the epidemic converges  to the DFE and thus dies  out, but starting with
an initial condition where the populations $H$ and $M$ (or simply $M$ or
$B$) is infected,  the epidemic converges to the unique  EE $g^*$, whose
support is $\{B,M,H\}$. 

In the example  of the zoonosis model~\ref{it:zm}, assuming  again that all
the atoms are supercritical, we deduce from Theorem~\ref{thI:cv_eq_max},
starting with an initial condition  $u(0)=h$, the epidemics converges in
long  time  towards the  equilibrium    whose support  is  the
support of $h$.

\medskip

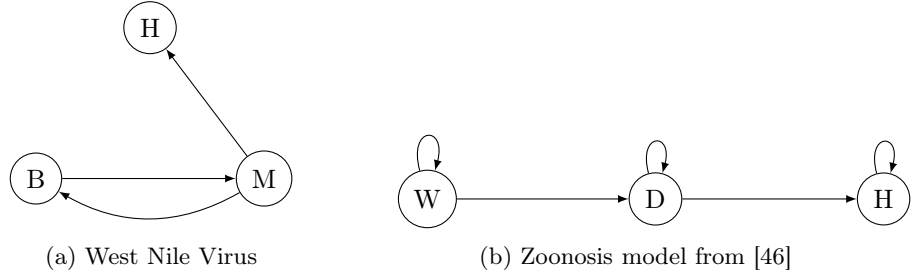
\begin{figure}
\centering
\begin{subfigure}[b]{.4\textwidth} \centering 
\begin{tikzpicture}
\node[draw,circle](1) at (0,0) {B};
\node[draw,circle](2) at (3,0) {M};
\node[draw,circle](3) at (1.5,2) {H};
\draw[->, >=latex] (2) to[out=210,in=330] (1);
\draw[->, >=latex] (1) to (2);
\draw[->, >=latex] (2) to (3);
\end{tikzpicture}
\caption{West Nile Virus}
\label{fig:WNV_graph}
\end{subfigure}
\begin{subfigure}[b]{.4\textwidth} \centering 
\begin{tikzpicture}
\node[draw,circle](1) at (0,0) {W};
\node[draw,circle](2) at (3,0) {D};
\node[draw,circle](3) at (6,0) {H};
\draw[->, >=latex] (1) to (2);
\draw[->, >=latex] (2) to (3);
\draw[->, >=latex] (1) to[loop above] (1);
\draw[->, >=latex] (2) to[loop above] (2);
\draw[->, >=latex] (3) to[loop above] (3);
\end{tikzpicture}
\caption{Zoonosis model from~\cite{royce20}}
\label{fig:RF_graph}
\end{subfigure}
\caption{Some examples of transmission graphs $G_K$}
\label{fig:T_and_others}
\end{figure}

\subsection{Assumptions and main results}\label{intro:sec:hp}

Recall the SIS model~\eqref{eq:intro_gen_SIS}-\eqref{eq:def-F} with
parameter $(T, \gamma,
\varphi)$.   We
shall 
consider  the following  assumptions  on the parameters.  For $p\in [1, +\infty ]$, let $L^p$ denote the usual Lebesque space
of    measurable    function    defined   on    the    measured    space
$(\Omega, \cg, \mu)$ endowed with the  $L^p$ norm $\norm{\cdot}_p$, and
$L^p_+$ the subset of $L^p$ of nonnegative functions. 

\begin{assum}\label{assum:1}
  The  measure  $\mu$ is  finite  and  non-zero;  the map $T$ is  a bounded
  linear map  on $L^\infty$ and there  exists $p \in (1,+\infty)$  and a
  finite constant  $C_p$ such  that for all $f\in L^\infty $:
\begin{equation}
   \label{eq:T-in-Linfty}
 \norm{Tf}_p\leq C_p  \, \norm{f}_p;
 \end{equation}  
the  function  $\gamma$
  belongs to  $L^\infty$ and $\gamma>0$ a.e.; and  the  function $\varphi : \R \rightarrow \R$
  is locally Lipschitz, nonnegative on $[0, 1]$ and $\varphi(1) = 0$.
\end{assum}

\begin{assum}\label{assum:2}
  Assumption~\ref{assum:1} holds; there exists  a finite constant $C'_p$
  such that for all
  $f\in L^\infty  $:
  \begin{equation}
   \label{eq:T/gamma-in-Lp}
 \norm{Tf}_\infty \leq C'_p  \, \norm{\gamma f}_p; 
 \end{equation}  
 and the map $\varphi$ is decreasing on $[0,1]$ with $\varphi(0) = 1$.
\end{assum}

The next two remarks are related to kernel operators. 
We also refer to Section~\ref{sec:T-and-others} for further  properties of the
operator $T$ induced by those two assumptions.

\begin{rem}[The operator $T$ is a kernel operator] 
  \label{rem:on-T}
  Assume  Condition~\eqref{eq:T/gamma-in-Lp} holds.   Since $\gamma$  is
  bounded, we  deduce that there  exists a finite constant  $C''_p$ such
  that   $\norm{Tf}_\infty   \leq   C''_p   \,   \norm{f}_p$   for   all
  $f\in  L^\infty  $,  and  also  that  Condition~\eqref{eq:T-in-Linfty}
  holds.  According to~\cite[Theorem~4.2]{schep}, we  deduce that $T$ is
  a kernel operator (and the kernel is indeed nonnegative by Theorem~1.3
  therein).
 \end{rem}

\begin{rem}[The SIS model from \cite{ddz-sis}]
  \label{rem:SIS-ddz}
  Recall the  definition of the kernel  operator $T_k$ for $k$  a kernel
  given   in  Section~\ref{subsec:intro:model}.    We  check   that  the
  SIS model from \cite{ddz-sis}, see Assumption~1 therein, satisfies our
  Assumption~\ref{assum:2}.  In~\cite{ddz-sis},  the measure $\mu$  is a
  probability measure on $\Omega$, the function $\gamma$ is positive and
  bounded,   and   the   mass-action   incidence  rate   is   associated
 to 
  $\varphi = 1-\id$. So the conditions on $\varphi$, $\gamma$ and $\mu$
  in Assumption~\ref{assum:2}  are clearly satisfied.  Therein,  we have
  $T = T_k$ for $k : \Omega^2 \rightarrow \R_+$ a kernel such that:
\begin{equation}
   \label{eq:hyp-k}
\sup_{x \in \Omega}\,  \int_\Omega \left(\frac{k(x,y)}{\gamma(y)}\right) ^q
\, \mu( \rd y) < +\infty,
\end{equation}
for some $q\in (1, +\infty )$.  It  is then elementary to check that the
conditions on  the operator $T=T_k$ from  Assumption~\ref{assum:2} are
satisfied  with $p \in (1,+\infty)$ given by $1/p + 1/q = 1$.
\end{rem}

\medskip

We now give our main results.  Recall that
$\Delta=\{f\in L_+^\infty \, \colon\, 1- f\in L^\infty _+\}$ is the
set of measurable functions taking their value a.e.\ in $[0, 1]$.
Proposition~\ref{prop:exists_max_sol} below asserts that under
Assumption~\ref{assum:1}, for any initial condition $h \in \Delta$,
Equation~\eqref{eq:intro_gen_SIS} has a unique global solution in
$L^\infty $ given by the semi-flow
$\left(\phi(t, h)\right)_{t \in \R_+}$ and that $\phi(t, h)$ belongs
to $\Delta$ for all $t\in \R_+$. The following result on long time convergence
appears below as~Theorem~\ref{th:cv_eq_max}
(see Section~\ref{sec:leb} below  for a precise definition of the convergence involved).
\begin{theoI}[Longtime behavior]
  \label{theo:cv-intro}
  Let $(T, \gamma, \varphi)$  satisfy Assumption~\ref{assum:2}.
  The semiflow $\phi$ always converges to an equilibrium: 
  for any initial condition  $h \in \Delta$, there exists
  $g\in \Delta$  such that $F(g) = 0$ and 
\begin{equation}
   \label{eq:cv-phi-g}
  \limess_{t \rightarrow + \infty} \phi(t, h) = g
  \quad\text{in $L^\infty $.}
\end{equation}
\end{theoI}
Under Assumption~\ref{assum:2}, we define  the basic reproduction number
$R_0$ as the spectral radius  $\rho(\tgi)$ of the power compact operator
$\tgi$   on  $L^\infty   $   given  by   $\tgi   f=  T(f/\gamma)$,   see
Lemma~\ref{lem:T}.  It  comes at no  surprise that if $R_0\leq  1$, then
the zero function $\zero$ is the  only equilibrium, so that all epidemic
disappear in the long  run, see Proposition~\ref{th:DDZ_cv_sf}. However,
if  $R_0>1$,  then  there  exists a  maximal  endemic  equilibrium,  say
$g^*\neq  \zero$, see  Theorem~\ref{prop:eq_crit_vacc}.  If  furthermore
$T$ is irreducible (which is equivalent to the existence and uniqueness,
up to  a scaling factor,  of $v\in L^\infty _+\setminus  \{\zero\}$ such
that $Tv =R_0\, v$ and that $v$ is positive), the maximal equilibrium is
the only endemic equilibrium and  $g$ in~\eqref{eq:cv-phi-g} is equal to
$g^*$  as soon  as  the initial  condition $h$  is  non-zero, see  again
Proposition~\ref{th:DDZ_cv_sf}.   Those   results  appear   already  in
\cite{ddz-sis} in a slightly less general framework for 
$R_0\leq 1$ or $T$ irreducible (or  quasi-irreducible).

The  main  result  of  the  paper is the description of  all the
endemic equilibria and their domain of
attraction: for any equilibrium $g$ we give  all  the initial
conditions $h\in  \Delta$ such that~\eqref{eq:cv-phi-g} holds.
To do so, we shall rely on the decomposition of the state space in atoms
associated to the operator $T$ given by Schwartz~\cite{schwartz_61}, see
also   our    previous   work~\cite{dlz},    which   is    recalled   in
Section~\ref{sec:decom_pos_op},   and  Section~\ref{intro:sec:unique_ee}
for  the  elementary case  where $\Omega$ is finite.
To  summarize,  a  measurable  set
$A \in \cg$  is invariant if the  support of the function  $T\un_A$ is a
subset of $A$ (up  to a set of zero measure); a set  is admissible if it
belongs to the $\sigma$-field generated by the invariant sets; the atoms
are the minimal admissible sets with positive measure (that is $A$ is an
atom  if $A$  is  admissible  with $\mu(A)>0$  and  if  $B\subset A$  is
admissible  then  either  $\mu(B)=0$  or $\mu(B)=\mu(A)$).   If  $T$  is
irreducible, then  $\Omega$ is an  atom. For  a measurable set  $A$, its
future $\cf(A)$  is the smallest invariant  set containing $A$ (up  to a
set of zero measure).  The set  of atoms (identifying atoms which differ
by  a  set of  zero  measure)  can be  endowed  with  an order  relation
$\preccurlyeq$: $A  \preccurlyeq B$ when $\cf(A)  \subset \cf(B)$ (where
the   inclusion   holds  up   to   a   set   of  zero   measure).    See
Section~\ref{sec:decom_pos_op} for further details.  We say that an atom
$A$  is supercritical  if the  spectral  radius of  the operator  $\tgi$
restricted to  $A$, denoted $R_0(A)$,  is strictly larger than  $1$. The
number of supercritical  atoms is finite; it is positive  if and only if
$R_0>1$, see~\cite{dlz}.   A \emph{supercritical antichain} is  a finite
set of supercritical atoms which  are pairwise unordered with respect to
$ \preccurlyeq$; we define its future as  the future of the union of its
atoms.  For example,  with two supercrtical atoms, say $A$  and $B$, the
supercritical antichains are $\emptyset$, $\{A\}$ and $\{B\}$, with also
$\{A, B\}$  if $A$ and  $B$ are unordered.   Notice that $R_0>1$  if and
only if there exists a non empty supercritical antichain.

We give a complete characterization of equilibria, see
Theorem~\ref{th:bij_antichains_eq} for a more complete statement.

\begin{theoI}[Equilibria and supercritical antichains are in bijection]
  \label{thI:bij_antichains_eq}
  If $(T,  \gamma, \varphi)$ satisfy Assumption~\ref{assum:2},  then the
  set  of supercritical  antichains and  the  set of  equilibria are  in
  bijection. Furthermore the support of the equilibrium associated to a
  supercritical 
  antichain is given by its future.
\end{theoI}

The empty supercritical antichain corresponds  to the DFE $g=\zero$. 
We  deduce  from  this  result  that if  $g$  and  $g'$  are  two
equilibria, then $g\leq g'$ if  and only if $\supp(g) \subset \supp(g')$
(up to a  set of zero measure) and $\supp(h)=\{h>0\}$  is the support of
the function $h$, see Corollary~\ref{cor:supp-equi=}.

To complete this theorem we fully describe basins of attraction. To state
the result, we denote by $T_A$  the projection  of $T$ on  a measurable set  $A$,
that is, the operator on  $L^\infty $ defined  by $T_Af=  \ind{A} T(\ind{A}  f)$ for
$f\in L^\infty $. Notice that if $T$ satisfies Assumption~\ref{assum:2},
so does  $T_A$. When this is  the case, we  say that $g$ is  the maximal
equilibrium  of $A$ when $g$ is the maximal equilibrium of the SIS model
with $T$ replaced by $T_A$.
Intuitively, 
from an  initial condition  $h\in \Delta$, the  epidemic  converges  to an
 equilibrium  which  depends only  on the  support of  the initial
condition; it is the same as the one starting from the  ``worst possible
case'' where the whole population in  $\cf(\supp(h))\subset \Omega$ is infected. 
 \begin{theoI}[Basins of attraction of equilibria]\label{thI:cv_eq_max}
  The limiting  equilibrium $g$  of an  epidemic with  initial condition
  $h\in   \Delta$  from   Theorem~\ref{theo:cv-intro}  is   the  maximal
  equilibrium of $\cf(\supp(h))$.
\end{theoI}
This result  appears below as~Theorem~\ref{th:cv_eq_max}.  It  is a full
generalization of Theorem 4.13  in~\cite{ddz-sis}, which only covers the
irreducible case $T$  where, if $R_0>1$, the endemic  equilibrium $g$ is
unique and all epidemics with initial condition $h\neq \zero$ converge to
$g$ in large time.   Proposition~\ref{prop:cv-unif} states that $\gamma$
times  the  epidemic converges  uniformly  to  $\gamma g$.   Thus,  when
$\essinf \gamma > 0$, the epidemic  converges uniformly to $g$, see also
Remark~\ref{rem:R0<1}   when   furthermore    $R_0<1$.    This   uniform
convergence is no longer true \emph{a priori} when $\essinf \gamma = 0$,
see Remark~\ref{rem:non-unif-cv} and Example~\ref{ex:unif-cv-g=0}.

\subsection{Model with an external disease reservoir}
\label{sec:intro-SISk}

We consider an infinite-dimensional SIS model with an external disease
reservoir, called SIS$\kappa$ model  in~\cite{reservoir}; it can be seen
as  an  extension  of   the  SIS  model~\eqref{eq:intro_gen_SIS}.   An
external  disease  reservoir is  a  particular  case of  environmentally
transmitted  diseases   where  the   population  of  pathogens   in  the
environment   is   assumed  to   be   constant   over  time,   see   for
example~\cite{gerba09, lanzas19}  and references therein.  See  also the
example  of the  West  Nile Virus,  where birds  and  mosquitoes form  a
reservoir that is not infected  by humans, see~\cite{bowman05}.  It also
encompasses some  SIS models with immigration  from~\cite{brauer01}, see
Remark~\ref{rem:immig} below.

Recall   $I(t)  \geq 0$  and  $S(t) \geq 0$  denote
respectively  the   number  of  infected  individuals   and  suscpetible
individuals      at      time       $t      \geq      0$.      According
to~\cite[Eq.~(2.1-2)]{reservoir},     the     corresponding     ordinary
differential  equations model,  including  infection  from the  external
disease reservoir, is given by:
\begin{equation}
  \label{eq:I-reservoir}
\left\{
  \begin{aligned}
    S'&= \mu_0 N - \mu_0 S - k \,\frac{SI}{N}- \kappa S  +\gamma I,\\
    I'&= k\, \frac{SI}{N} + \kappa S - (\mu_0 +\gamma) I,
  \end{aligned}
    \right.
  \end{equation}
  where $N=I+S$  is the total  population, $\mu_0$ is the  healthy birth
  rate  and  the common  death  rate  of  the susceptible  and  infected
  populations, $\kappa S$  is the rate of disease  transmission from the
  reservoir, with $\kappa>0$.   Notice the total size  population $N$ is
  constant in time. 

In an inhomogeneous setting,
with the measured space of types  $(\Omega, \cg, \mu)$ and  $\mu(\Omega)\in (0,
+\infty )$, the proportion of infected individuals among the
individuals with feature $x$ is given by 
$u(t,x)=I(t,x)/N(x)$,  where  $I(t,x)$ denotes  the number  of
infected individuals  with feature  $x\in \Omega$  at time $t\in
\R_+$ and $N(x)$ the total size of the population with feature $x \in
\Omega$, assumed constant over time. In the inhomogeneous SIS$\kappa$ model
inspired by~\eqref{eq:I-reservoir},  the function 
$u=(u(t,x))_{t\in \R_+, x\in \Omega}$ is 
solution   in $L^\infty $ of  the ODE:
\begin{equation}
  \label{eq:intro:gen_SIS_immigration}
\left\{\begin{array}{l}
u' = F_\kappa(u), \\ 
u(0) = h,
\end{array}\right.
\end{equation}
with initial condition $h\in L^\infty $ and:
\begin{equation}\label{eq:intro:intuition_eq_immi}
F_\kappa(u)=  \varphi(u) (Tu+\kappa) - \gamma u,
\end{equation}
where    $\varphi$   is    a   continuous    function   on    $\R$   and
$\kappa\in    L^\infty    _+$. 
The particular case of SIS$\kappa$ model given by~\eqref{eq:I-reservoir}
corresponds to $\Omega=\{\omega\}$, $\mu$ a Dirac mass at $\omega$, $T$
the multiplication operator by $k$, $\gamma$ and $\kappa$ constant functions, and $\varphi=1-\id$.

\medskip

This model can be related to SIS model with immigration, see the following remark.

\begin{rem}[SIS model with immigration]
  \label{rem:immig}
  For   the   homogeneous   population,    we   link   the   SIS$\kappa$
  model~\eqref{eq:I-reservoir}  with  the  SIS  model  with  immigration
  of~\cite[Eq.~(1)]{brauer01}. Assume initially that the total population $N(t)$ is not necessarily constant over time $t$. Let $A \geq  0$ be the  immigration rate,  $p \in  [0,1]$ the
  proportion of infected  individuals among the immigrants, and  $d > 0$
  the death rate among the population.  All those parameters are assumed
  to be  constant over  time.  We  assume that  the epidemic  induces no
  death (that is $\alpha =  0$ in \cite[Eq.~(1)]{brauer01}) and that the
  incidence rate is the standard  mass-action.  Then, the SIS model with
  immigration given in~\cite[Eq.(11)]{brauer01} reduces to:
\begin{equation}\label{eq:brauer_immi}
\left\{
    \begin{aligned}
   I' &= pA + k\, \frac{(N-I)I}{N}   - (d+\gamma)I, \\
  N' &= A - d N.
    \end{aligned}
    \right.
\end{equation}
Since $\lim_{t\rightarrow\infty  } N(t)=A/d$, and  since we
are interested  in the long  time equilibrium,  it is natural  to assume
that $N$ start  at its equilibrium, that  is $N(0)=A/d$, so
that the population size is constant over time.
In this case, Equation~\eqref{eq:brauer_immi}  with $u(t) = I(t)/N(0)$
reduces to:
\begin{equation}
   \label{eq:imm-u}
u' = (1-u)(ku+pd) - (\gamma + (1-p)d)u.
\end{equation}
The same arguments applied to an inhomogeneous population would lead to
a similar multi/infinite-dimensional ODE with $u(t)$ replaced by a
function $u(t,x)$, with $x\in \Omega$ the set of features, and $ku$
replaced by $Tu$ with $T$ the transmission operator, so that
\eqref{eq:imm-u} becomes:
\[
u' = (1-u)(Tu+pd) - (\gamma + (1-p)d)u.
\]
This            corresponds to the SIS$\kappa$ model~\eqref{eq:intro:gen_SIS_immigration}-\eqref{eq:intro:intuition_eq_immi}     with
$\varphi= 1-\id$,  $\kappa = pd$ and $\gamma$ replaced by
$\gamma + (1-p)d$. In conclusion the SIS model with immigration and the
SIS model with an external disease reservoir lead to the same ODE.

\end{rem}

In  Proposition~\ref{prop:sol-reservoir} and
Corollary~\ref{cor:eq-reservoir}, we prove that the SIS$\kappa$ model
with reservoir of~\eqref{eq:intro:gen_SIS_immigration} can  be analyzed
using the classical SIS model~\eqref{eq:intro_gen_SIS} by adding a new
element $\r$ to the set of features $\Omega$ corresponding to the
reservoir. In particular we provide a full description of the equilibria and
their domain of attraction for the  SIS$\kappa$ model.

\subsection{Discussion on the incidence rate}\label{sec:intro_KMK_LMA}

In this section, we discuss different models for the infection rate, and
more precisely for the function $\varphi$ in~\eqref{eq:def-F}.

In an homogeneous population, Ross~\cite{ross1} considered the so called
law of mass action $\beta SI/N$ (which corresponds to $\varphi = 1-\id$ in
the  SIS  model):  the  incidence  rate  is
proportional to the product of the proportion of susceptible individuals
and the  proportion of  infected individuals.   According to  Wilson and
Worcester  \cite{wilson45},  it  corresponds   to  the  assumption  that
infected  individuals are  mixing  uniformly with  the susceptible  ones
throughout the population, see also Heesterbeek~\cite{heesterbeek05} for
an historical review.  Some epidemic models introduced in the literature
replace  the law  of  mass action  by various  incidence  rates, see  in
particular  the  survey  McCallum, Barlow  and  Homeo~\cite{mccallum01}.
Concerning    the    function    $\varphi$,    Assumptions~\ref{assum:1}
and~\ref{assum:2} below reduce to:
\begin{equation}
  \label{eq:hyp-phi}
  \text{$\varphi$ is locally Lipschitz on $\R$,
decreasing on $[0,1]$ with   $\varphi(1) =
0$ and $\varphi(0) = 1$.}
\end{equation}
  In the 
examples below from the literature, the set $\Omega$ is a singleton and the
transmission operator $T$ is thus a  constant, which is assumed to be
positive; so the condition  $\varphi(0) = 1$  is a normalization convention on
the (constant) operator $T$ and could be replaced here by the more relevant condition $\varphi(0)>0$. 

\begin{enumerate}[(i)]
\item London and Yorke~\cite{yorke73} consider the incidence rate $\beta
  SI(1 -cI)$, that is: 
  \[
    \varphi(u)  = (1-u)(1-au)\quad\text{with}\quad  a>0
  \]
  for measles  epidemic (in  New York  City and  Baltimore from  1928 to
  1972)  in  order  to  eliminate the  systematic  differences  on  data
  observed between years  with many cases and years  with relatively few
  cases; however they do not provide a biologic or physical argument for
  such modification.  Notice that by  considering $u/a$ instead  of $u$,
  one    can    assume   that    $a    \leq    1$,   in    which    case
  Condition~\eqref{eq:hyp-phi} holds.

\item \label{intro:incidence_rate:rose} We recall that in the SIR model,
  once  infected, the  individuals  recover with  a permanent  immunity.
  Rose et al.~\cite{rose21} incorporate in  the SIR model (with constant
  population  $N=S+I+R$)   a  population-level  heterogeneity   for  the
  infection susceptibility given by  the gamma probability distribution;
  in \cite[Section 4]{rose21} they   consider  the   incidence   rate
  $\beta I S^\alpha$. Using data  from the 2009 H1N1 influenza outbreak,
  they observe that the higher-order models are more consistent with the
  data  than the  case  $\alpha=1$.  In our  setting,  this model  would
  correspond  to   the  following  function  $\varphi$   with  satisfies
  Condition~\eqref{eq:hyp-phi}:
  \[
    \varphi(u) = (1-u)^\alpha
    \quad\text{with}\quad
    \alpha\geq 1.
  \]

\item Capasso and  Serio~\cite{capasso78} study a SIR  model (with
  constant population $N=S+I+R$)  taking into
  account saturation  and ``psychological'' effects.  They  consider the
  incidence rate $g(I)S$. 
In our  setting,  this model  would
  correspond  to:
\[
    \varphi(u) = (1-u)\, \frac{g(u)}{u},
  \]
  where the conditions  on $g$ translated into  our framework correspond
  to:  the function  $g$ is  defined on  $[0,1]$, nonnegative,  bounded,
  differentiable  with  $g'$ bounded and such  that $g(0) = 0$,  $g'(0)=1$ and
  $g(u) \leq u$  on $\R_+$.  Under those assumption,  the function 
  $\varphi$   is  Lipschitz   on   $[0,1]$,   with  $\varphi(1)=0$   and
  $\varphi(0)=1$.  However  the monotonicity  condition on  $\varphi$ on
  $[0,1]$, which amounts  to $g(u)\geq u(1-u) \, g'(u)$ on  $[0, 1]$, is
  not   satisfied   in   general.

  Notice   that
  Condition~\eqref{eq:hyp-phi} is indeed satisfied for the following
 functions $g$, where $c>0$:
 $u/(1+cu)$ in~\cite[Section~6]{capasso78},
$(1-\exp(-cu))/c$ in~\cite{kolokolnikov21} on SIR model for
Covid19     outbreak,  and $\log(1+cu)/c$ in Table 1 of the
survey~\cite{mccallum01} on pathogen transmission models. They
respectively correspond to:
\[
     \varphi(u) = \frac{1-u}{1+cu}, 
\quad 
\varphi(u) =  (1-u)\,  \frac{1-\exp(-cu)}{cu}
\quad\text{and}\quad
\varphi(u)=(1-u) \, \frac{\log(1+cu)}{cu}\cdot
  \]

 \item We recall that in the  SIRS model, once infected, the individuals
   recover with a temporary immunity. To exhibit qualitatively different
   dynamical behaviors,  Liu, Lewin and Iwasa~\cite{liu86}  introduced a
   SIRS model  (with constant population $N=S+I+R$)  where the incidence
   rate is  given by $  I H(I,S)$  for some differentiable  function $H$
   such that $H(I,0)=0$ and $\partial_S  H>0$ for all $I>0$.  The latter
   condition reflects  the biologically  intuitive requirement  that the
   incidence  rate   be  an  increasing   function  of  the   number  of
   susceptibles. In our setting, this model would correspond to:
   \[
     \varphi(u)= H(u, 1-u),
   \]
   with $\varphi$ differentiable and  $\varphi(1)=0$.
Notice that $\varphi$ is decreasing on $[0,1]$ if  $\partial_S H>
\partial_I H$. 
The authors consider the  particular case
 $\varphi(u) =  (1-u)^\alpha \, u^{\beta -1}$ with $\alpha, \beta >
  0$.  Condition~\eqref{eq:hyp-phi}  holds for $\beta = 1$ and
  $\alpha\geq 1$, which is already considered in
  Point~\ref{intro:incidence_rate:rose}.

\end{enumerate}

\section{Notations}\label{sec:notation}

\subsection{Ordered set}

Let $(E, \leq )$ be a (partially) ordered set.  Whenever it exists,
the \emph{supremum} of $A\subset E$, denoted by $\sup(A)$, is the
least upper bound of $A$: for all $x\in A$, $x \leq \sup(A)$ and if
for some $z\in E$ one has $x\leq z$ for all $x\in A$, then
$\sup(A) \leq z$.  A collection $(x_i)_{i \in \mathcal{I}}$ of
elements of $E$ is an \emph{antichain} if for all distinct
$i, j \in \mathcal{I}$, the elements $x_i$ and $x_j$ are not
comparable for the order relation.

\subsection{Banach space and Banach lattice}
\label{sec:banach}
Let $(X, \norm{\cdot})$ be a complex Banach space not reduced
to~$\{0\}$. An operator $T$ on $X$ is a  bounded  linear (and thus
continuous) map from $X$ to itself. If $Y \subset X$ is a subspace of $X$ such that $T(Y) \subset Y$, we denote $T|_Y$ the restriction of $T$ to the subspace $Y$, that is an operator on the Banach space $(Y, \norm{\cdot})$.
The  operator norm of $T$ is given by:
\begin{equation}
   \label{eq:def-norm-T}
  \norm{T}_X= \sup\left\{ \norm{Tx}\, \colon\, x \in X \text{ s.t. }
  \norm{x} = 1\right\},
\end{equation}
its                              spectrum                             by
$\spec(T) =  \{\lambda \in \C\, \colon\,  T - \lambda \id  \text{ has no
  bounded inverse}\}$, where $\id$ is  the identity operator on $X$.  If
$\lambda \in \C$  and $x\in X\priv{0}$ satisfy $Tx=\lambda  x$, then the
element $x$  is an eigenvector  of $T$  and $\lambda$, which  belongs to
$\spec(T) $,  is an eigenvalue  of $T$.  The  spectral radius of  $T$ is
defined by (see \cite[Theorem~18.9]{rudin87}):
\begin{equation}
   \label{eq:def-rho}
  \rho(T) = \sup\left\{ |\lambda| \,\colon\, \lambda \in \spec
    (T)\right\} =\lim_{n\rightarrow \infty } \norm{T^n}_X^{1/n}.
\end{equation}
By convention, we set $T^0=\id$.
The spectral radius
is commutative in the sense that if $T$ and $S$ are two operators on
$X$, we have:
\begin{equation}\label{eq:spec_rac_commute}
\rho(TS) = \rho(ST).
\end{equation}
We define the spectral bound of the operator $T$ by:
\begin{equation}
  \label{eq:def-s(T)}
 s(T) = \sup \{\Reel(\lambda) \,\colon\, \lambda \in \spec(T)\}.
\end{equation}

\medskip

Let $X^\star$ denote  the (continuous or topological)  dual Banach space
of $X$, that is the set of  all the continuous linear forms on $X$.  For
$x\in X$, $x^\star\in X^\star$, let  $\langle x^\star, x \rangle$ denote
the duality product and the norm of $x^\star$ in $X^\star$ is defined by
$\norm{x^\star}=\sup\{\langle    x^\star,     x    \rangle\,    \colon\,
\norm{x}=1\}$.   For an  operator $T$,  the dual  operator $T^\star$  on
$X^\star$                 is                  defined                 by
$\langle T^\star x^\star, x \rangle=\langle x^\star, Tx \rangle$ for all
$x\in   X$,    $x^\star\in   X^\star$.    It   is    well   known   that
$\norm{T^\star}_{X^\star}=\norm{T}_X$ and $\spec(T^\star) = \spec(T)$.

\medskip

An ordered real Banach space $(X, \norm{\cdot}, \leq )$ is a real
Banach space $(X, \norm{\cdot})$ with an order relation $\leq$.  For
any $x \in X$, we define $|x| = \sup(\{x, -x\})$ the supremum of $x$
and $-x$ whenever it exists.  Following \cite[Section~2]{schaefer_74},
the ordered Banach space $(X, \norm{\cdot}, \leq)$ is a \emph{Banach
  lattice} if:
\begin{enumerate}
\item For  any $x, y, z  \in X, \lambda \geq  0$ such that $x  \leq y$, we
  have $x + z \leq y + z$ and $\lambda x \leq \lambda y$.
\item For any $x, y \in X$, there exists  a supremum of $x$ and $y$ in $X$.
\item For any $x, y \in X$ such that $|x| \leq |y|$, we have
  $\norm{x} \leq \norm{y}$.
\end{enumerate}

Let  $(X, \norm{\cdot},  \leq)$ be  a  real Banach  lattice.  We  denote
$X_+  =  \{x  \in X  \,\colon\,  x  \geq  0  \}$ the  positive  cone  of
$X$.  Recall it is a closed
set.
We shall also consider the
dual cone $X^\star_+=\{x^\star\in  X^\star\, \colon\, \langle x^\star, x
\rangle\geq 0 \text{ for all } x\in X_+\}$. A  linear map  $T$ on  $X$ is  \emph{positive} if
$T(X_+)  \subset  X_+$.   According  to~\cite[Theorem~4.3]{aliprantis06}
positive  linear maps  on  Banach  lattices are  bounded  (and thus  are
operators).

If  $S$ and $T$ are  two operators on $X$,
we write  $T \leq S$ if the operator $S - T$ is positive.
If the operators $T, S$ and $S-T$ are positive, then we have, see \cite[Theorem 4.2]{marek70}:
\begin{equation}\label{eq:spec_rad_croissant}
\rho(T) \leq \rho(S).
\end{equation}

\medskip

Any real Banach  lattice $X$  and any  operator $T$ on  $X$ admits  a natural
complex  extension.  The  spectrum  of  $T$ will  be  identified as  the
spectrum of its complex extension and denoted by $\spec(T)$, furthermore
by \cite[Lemma  6.22]{abramovich02}, the spectral radius  of the complex
extension             is            also             given            by
$\lim_{n\rightarrow \infty } \norm{T^n}_X^{1/n}$, with $\norm{\cdot}_X$
still defined by~\eqref{eq:def-norm-T}. 
Moreover, by \cite[Corollary 3.23]{abramovich02}, if $T$ is positive
(seen as an operator on the real Banach lattice $X$), then $T$ and its complex
extension have the same norm.

\subsection{Lebesgue spaces and essential limits}
\label{sec:leb}

Let  $(\Omega, \cg,  \mu)$  be   a  measured  space  with $\mu$ a
$\sigma$-finite measure.       For      any
$\mathcal{A}  \subset \cg$,  we denote  by $\sigma(\mathcal{A})$
the  $\sigma$-field  generated by  $\mathcal{A}$.   If  $f, g$  are  two
real-valued  measurable   functions  defined   on  $\Omega$,   we  write
$f  \leq g$  a.e.\ (resp.   $f =  g$ a.e.)   when $\mu(\{f  > g\})  = 0$
(resp. $\mu(\{f \neq g  \}) = 0$), and denote $\supp(f  )= \{f \neq 0\}$
the support of  $f$.    We  say that  a  real-valued
measurable function $f$ is nonnegative when  $f \geq 0$ a.e., we say
that    $f$   is    positive,    denoted   $f    >    0$   a.e.,    when
$\mu(\{f \leq 0\}) = 0$, and we say that $f$ is bounded if there exists $M \geq 0$ such that $|f| \leq M$ a.e.. 
If  $A, B \subset \Omega$ are measurable sets,
we write $A \subset B$ a.e.\ (resp $A = B$ a.e.) when $\un_A \leq \un_B$
a.e.\ (resp.  $\un_A = \un_B$ a.e.). Let $L^0(\Omega, \cg, \mu)$, simply
denoted $L^0$,   be the
set of $[-\infty , + \infty ]$-valued
measurable  functions defined  on  $\Omega$, where  functions which  are
a.e.\  equal  are   identified. 
The elements  $\un$ and $\zero$ of $L^0$ denote the
functions which  are a.e.\ equal  respectively to 1  and to 0.   For the
sake of clarity, we will omit to write a.e.\ in the proofs.

\medskip

Let $(f_t)_{t\in T}$ be a family of measurable functions defined on
$\Omega$ taking values in $[-\infty , +\infty ]$. We recall that
$f^*=\esssup_{t\in T} f_t$ is a measurable function such that
$f_t\leq f^*$ a.e.\ for all $t\in T$ and if $f$ is measurable function
such that $f_t\leq f$ a.e.\ for all $t\in T$ then $f ^*\leq f$ a.e.\
(if $T$ is at most countable, then one can take
$\esssup_{t\in T} f_t=\sup_{t\in T} f_t$).  We now consider $T=\R_+$.
Let $(f_t)_{ t\in \R_+}$ be a non-decreasing sequence, in the sense
that for all $t\leq s$ we have $f_t\leq f_s$ a.e., then if
$(t_n)_{ n\in \N}$ is a sequence converging to $+\infty $, we have
that the sequence $(f_{t_n})_{ n\in \N}$ converges a.e.\ towards
$f^*=\esssup_{t\in \R_+} f_t$, and thus we shall simply write
$f^*=\limess_{t\rightarrow +\infty } f_t$. We leave to the reader the
definition of $\essinf$ and the limit of a non-increasing sequence of
measurable functions.  For the family $(f_t)_{ t\in \R_+}$ , we
consider $f^*_t=\esssup_{s\geq t} f_s$ for all $t\in \R_+$, and get
that the sequence $(f^*_t)_{t\in \R^+}$ is non-increasing and write
$\limesssup_{t\rightarrow \infty } f_t= \limess_{t\rightarrow \infty }
f^*_t$.  We define in a similar way
$\limessinf_{t\rightarrow \infty } f_t$.  Notice that if $g_t=f_t$
a.e.\ for all $t\in \R_+$, then the essential supremum/infimum limits
of $(f_t)_{t\in \R^+}$ and $(g_t)_{t\in \R^+}$ are a.e.\ equal.
Therefore, the essential supremum/infimum limits of sequences is well
defined on the space $L^0$.  We say the sequence of functions
$(f_t)_{ t\in \R_+}$ in $L^0$ essentially converges if
$\limesssup_{t\rightarrow \infty } f_t=\limessinf_{t\rightarrow \infty
} f_t$ in $L^0$, and write $\limess_{t\rightarrow +\infty } f_t$ for
this common limit (which is an element of $L^0$).  When considering
$T=\N$ instead of $T=\R_+$, the analog of the essential convergence is
the a.e.\ convergence, that is the usual convergence in $L^0$.

\medskip

For a measurable function $f$, we write
$\mu(f)=\int f\, \rd \mu=\int_\Omega f(x)\, \mu(\rd x)$ the integral
of $f$ with respect to $\mu$ when it is well defined. When $f$ is
measurable and a.e.\ finite and nonnegative, we denote $f \mu$ the
measure on $(\Omega, \cg)$ defined by $f \mu(A) = \mu(\un_A f)$ for
any measurable set $A$.  For $p \in [1, +\infty]$, the Lebesgue space
$L^p(\Omega, \cg, \mu)$ is the set of all real-valued measurable
functions $f\in L^0$ defined on $\Omega$ whose $L^p$-norm,
$\norm{f}_p = \mu(|f|^p)^{1/p}$ if $p<+\infty $ and
$\norm{f}_\infty = \esssup(|f|)$ if $p=+\infty $, is finite.  When
there is no ambiguity we shall simply write $L^p(\Omega)$, $L^p(\mu)$
or $L^p$ for $L^p(\Omega, \cg, \mu)$.  The Banach space $L^p$ endowed
with the usual order $f\leq g$, that is $\mu(\{f>g\})=0$, is a Banach
lattice.  The positive cone $L^p_+$ is the subset of $L^p$ of
nonnegative functions; it is normal (as the norm $\norm{\cdot}$ is
monotonic, that is, $0\leq f\leq g$ implies
$\norm{f}_p\leq \norm{g}_p$, see~\cite[Proposition~19.1]{deimling85})
and reproducing (that is, $L^p_+ -L^p_+=L^p$).  Since the supports of
two functions which are a.e.\ equal are also a.e.\ equal, we get that
the support of $f\in L^p$ is well defined up to the a.e.\ equality; it
will still be denoted by $\supp(f)$.  For $p\in [1, +\infty )$, the
dual of $L^p$ is $L^q$ where $1/p+1/q=1$, with the duality product
$\langle g, f \rangle = \int fg\, \rd \mu$ for $f \in L^p$ and
$ g \in L^q$ (for $p=1$, we use that the measure $\mu$ is
$\sigma$-finite).

\medskip

For any
$f \in L^\infty$, we denote by  $M_f$ the multiplication by $f$, which
can be seen as 
an operator on $L^p$ for $p \in [1,+\infty]$.  For $A \in \cg$  a
measurable set, we denote:
\begin{equation}
  \label{eq:def-MA}
  M_A=M_{\un_A}.
\end{equation}
Let $T$ be an operator on $L^p$.  
The projection 
 of  $T$  on
$A$, denoted $T_A$, is the operator defined  by:
\begin{equation}
   \label{eq:def-TA}
   T_A=M_A \, T\,  M_A,
\end{equation}
and, if $\mu(A)>0$, we denote by  $T|_A$ the restriction of the operator
$T_A$ to $L^p(A)$, where the set $A$  is endowed with the trace of $\cg$
on $A$ and the measure $\mu|_A (\cdot)=\mu(A\cap \cdot)$.

We now assume that $\mu(\Omega)>0$, so that $L^p(\Omega)$ is not reduced
to  a singleton.  When there  is no  ambiguity on  the operator  $T$, we
simply write $\rho(A)$  for the spectral radius of $T_A$  (and of $T|_A$
when  $\mu(A)>0$).  In  particular, we  have $\rho(\Omega)=\rho(T)$  and
$\rho(A) = 0$ if $\mu(A) = 0$.  If the operator $T$ is positive, we also
have that:
\[
  A \subset B \quad \Longrightarrow \quad \rho(A) \leq  \rho(B).
\]

\subsection{Decomposition of positive operators on $L^p$, with $p\in (1,
+\infty )$}\label{sec:decom_pos_op}
Recall the measure $\mu$ is $\sigma$-finite and non-zero.
We recall the atomic decomposition from Schwartz~\cite{schwartz_61} of a
positive operator $T$  on $L^p$, see also~\cite{dlz}.   A measurable set
$A \in  \cg$ is $T$-\emph{invariant}, or simply invariant when there is
no ambiguity,   if $M_{A^c} T  M_A = 0$, which, see~\cite[Eq.~(7)]{dlz}, is equivalent
to:
\begin{equation}\label{eq:inv_supp_plus_petit}
  \langle g , Tf   \rangle=0.
\end{equation}
for all $f \in L^p_+$ and $  g \in L^q_+$ such that $\supp(f) \subset A$
a.e.\  and   $\supp(g)  \subset   A^c$  a.e..    The  operator   $T$  is
\emph{irreducible}  if  its  only  invariant sets  are  a.e.\  equal  to
$\Omega$ or $\emptyset$.  A measurable  set $A$ with positive measure is
\emph{irreducible} if  the operator  $T|_A$ on $L^p(A)$  is irreducible.
The \emph{future}  of a set $A  \subset \Omega$, denoted $\cf(A)$,  is the
smallest invariant set  that contains $A$.  If $\cc$ is  an at most countable collection of subsets  of $\Omega$,  then  we denote  by $\cf(\cc)$  the
future of the union of the elements of $\cc$:
\begin{equation}
   \label{eq:F(C)}
   \cf(\cc)= \cf\left(  \bigcup_{A\in \cc}  A\right) =  \bigcup_{A\in \cc}  \cf(A),
 \end{equation}
 where the last equality is~\cite[Lemma 3.13]{dlz}.

A set
is \emph{admissible}  if it belongs  to the $\sigma$-field  generated by
the invariant  sets. An \emph{atom} is  a minimal admissible set  with a
positive  measure  (that is,  $A$  is  an  atom  if $A$  is  admissible,
$\mu(A)>0$ and if $B$ is an  admissible set such that $B\subset A$ a.e.\
then a.e.\ $B=\emptyset$  or $B=A$), and we identify two  atoms that are
a.e.\  equal.  In  particular, if  the set  $A$ is  an atom  and $B$  is
admissible then we have:
\[
  \mu(A \cap B)>0 \implies A\subset B \quad\text{a.e.}.
\]
According to~\cite[Theorem~1]{dlz}, 
a measurable set  is admissible and irreducible with positive measure if
and only if it is an atom. 
Since the  atoms have  positive measure,  we get that  the set  of atoms  (up to
the a.e.\ equality),  $\atom$,   is at
most countable.
We shall also consider the (at most countable) set of non-zero atoms:
\[
  \anz=\{ A \in \atom\, \colon\, \rho(A) >0\}.
\]

The relation  $\preccurlyeq$ on $\atom$,
defined  by  $A  \preccurlyeq  B$  if  $\cf(A)  \subset  \cf(B)$  a.e.\  (or
equivalently $A \subset \cf(B)$ a.e.), is  an order relation.  We end this
section by noticing that antichains  of atoms are characterized by their
future.
\begin{lem}[Antichains with  same future]\label{lem:antichain_fut_eq}
  Let $\cc$ and $\cc'$ be two antichains of atoms. Then, we have:
  \[
    \cf(\cc)=\cf(\cc')\quad  \Longleftrightarrow \quad \cc=\cc'.
  \]
\end{lem}  

\begin{proof}
  Assume that $  \cf(\cc)=\cf(\cc')$. 
  Consider an atom  $A\in \cc$. Since
  $A\subset \cf(\cc) = \cf(\cc')$, 
  there exists $A' \in \cc'$ such that we have
  $\mu(A \cap \cf(A')) > 0$, which implies $A \preccurlyeq A'$ as $A$ is
  an atom. Conversely there exists $B\in \cc$ such that
  $A'\preccurlyeq B$, and by transitivity $A\preccurlyeq B$.  Since
  $\cc$ is an antichain, we obtain  $A=B$ and thus $A=A'$ is  an element
  of $\cc'$.
  The reverse implication is trivial by~\eqref{eq:F(C)}. 
\end{proof}

  \subsection{Power compact operators on $L^p$}

A linear map  $T$ on a Banach space is \emph{compact} if the image of the
unit ball is relatively compact; it is then bounded. 
An  operator $T$  on a  Banach space  is \emph{power  compact} if  there
exists  $k  \in \N^*$  such  that  $T^k$  is  compact.  We  recall  some
well-known properties of power compact  operators, see
\cite{dunford88,konig86} for  instance.

\begin{lem}[Spectrum of power compact operators]\label{lem:op_pc}
Let $T$ be an operator on a Banach space. 
\begin{enumerate}[(i)]
\item \label{th:item:T^*_pc}
  The operator  $T^\star$ is power compact if
  and only if the operator $T$ is power compact.
\item \label{th:item:spT_countable}
  If $T$  is power compact,  then the
  set $\spec(T)$  is at  most countable and  has no  accumulation points
  except possibly $0$ (it is thus totally disconnected).
\end{enumerate}
\end{lem}

It  is well  known  that the  spectral radius  (and  more generally  the
spectra) is a  continuous function on the set of  compact operators with
respect         to         the        operator         norm,         see
\cite[Theorem~11]{newburgh1951varspectra}.   We  shall   however  use  a
weaker  result   from  Anselone~\cite{anselone}.   We  say a  family  $\cv$  of
operators on $X$  is \emph{collectively compact} if $\bigcup  _{V\in \cv} V(B)$
is relatively compact, where $B$ is the unit ball of $X$.  The following
result  is  a direct  consequence  of  Proposition~4.1 and  Theorem~4.16
in~\cite{anselone}.
 
\begin{lem}[Collectively compact operators]
  \label{lem:coll-K}
  Let $I$  be an interval  of $\R$  and $(V_t)_{t\in  I}$ be a  family of
  collectively compact operators on a Banach  space $X$.  If $t\in I$ is
  such  that  $\lim_{s\rightarrow  t}   \norm{(V_s  -V_t)x}=0$  for  all
  $x\in X$, then we have $\lim_{s\rightarrow t}\rho(V_s)=\rho(V_t)$.
\end{lem}

We give a result on compact operators in Lebesgue space. 
Recall that  $\mu$ is a non-zero $\sigma$-finite  measure    on
$\Omega$,  and that   $L^p$    denote
$L^p(\Omega, \cg, \mu)$.

\begin{lem}[On compactness]
  \label{lem:hp_2_cpct}
  Let  $p \in  (1,+\infty)$.
  A positive operator on $L^p$  which 
 is dominated by a compact operator  is  compact.
\end{lem}

\begin{proof}
   Notice    that   $L^p$    has   an    order   continuous    norm   for
  $p\in [1, \infty )$, see~\cite[Definition~4.7]{aliprantis06}, that is,
  according to~\cite[Theorem~4.9]{aliprantis06}, if $(f_n)_{n\in \N}$ is
  a     non-decreasing     sequence     of     $L^p_+$     such     that
  $\sup_{n\in          \N}          f_n=f\in         L^p$,          then
  $\lim_{n\rightarrow \infty } \norm{f_n-f}_p=0$.
  In particular, when $p \in (1,+\infty)$, the dual of $L^p$, that is isomorphic to $L^q$ with $1/p + 1/q = 1$, also has an order continuous norm.
 The lemma is then  a
  direct consequence of~\cite[Theorem~5.20]{aliprantis06}.
\end{proof}

We recall  in our framework some classical
results, see~\cite[Theorem~6.2 and Lemma~6.5]{dlz}.

\begin{theo}
  \label{theo:rappel}
Let $T$ be a positive power compact operator on $L^p$ with $p \in
(1,+\infty)$. 
\begin{enumerate}[(i)]
\item \textbf{Krein-Rutman.} \label{thm:KR}
  If $\rho(T)$  is positive then $\rho(T)$ is
  an  eigenvalue of
  $T$, and there exists  a corresponding nonnegative right eigenfunction
  denoted 
  $v_T$.\label{th:item:KR}
  \item \textbf{de Pagter.} If $T$ is irreducible then $\rho(T)$ is positive unless $T=0$
    and $\dim(L^p)=1$, that is, if $A$ is measurable then either
    $\mu(A)=0$ or $\mu(A^c)=0$.\label{th:item:dP}
    \item \textbf{Schwartz.}
      We have that for any admissible set $A$:
    \begin{equation}
 \label{eq:R0=max}
  \rho(A)=\max_{B\in \anz, \, B\subset A} \rho(B).
\end{equation}
    \end{enumerate}
  \end{theo}

Following   \cite[Lemma~4.2]{ddz_vacc_crit},  we   now   state   a 
technical result based on Collatz-Wielandt inequality and the
Krein-Rutman theorem,   giving a bound on  the spectral radius given  a strict
supersolution  of  the  eigenvalue  equation.
\begin{lem}[Supersolutions and spectral radius]
  \label{lem:rad>1}
  Let $S$ be a positive operator on $L^p$ with $p \in (1, + \infty)$.
  If there exists $\lambda>0$ and a non-null nonnegative function $v \in L^p_+$ such that
  $Sv \geq \lambda v$, then:
  \[
    \rho(S)\geq \rho(S_A) \geq
    \lambda,
  \]
   where $A=\supp(v)$. Furthermore, if $S$ is power compact, then we have:
  \begin{enumerate}[(i)]
  \item \label{lem:item:superequal}
    If $Sv= \lambda v$ on $A$, then $\rho(S_A) = \lambda$.
  \item \label{lem:item:superstrict}
    If  $Sv - \lambda v$ is positive on $A$, then  $\rho(S_A) >
    \lambda$. 
  \end{enumerate}
\end{lem}

\begin{proof}
  We       first      note       that      $\rho(S)\geq       \rho(S_A)$
  by~\eqref{eq:spec_rad_croissant}.     Multiplying    the    inequality
  $Sv\geq \lambda  v$ by $\ind{A}$ yields  that $S_A v \geq  \lambda v$.
  The     fact     that     $\rho(S_A)\geq    \lambda$,     and     thus
  $\rho(S)  \geq  \lambda$,   is  then  a  direct   consequence  of  the
  Collatz-Wielandt     inequality     \cite[Propositions     2.1     and
  2.2]{forsternagy89}.

  We    now    assume    that    $S$   is    power    compact.     Since
  $(S_A)^\star  =   (S^\star)_A$,  we   shall  denote  them   simply  by
  $S^\star_    A$.     Since    $S_A\leq     S$,    we    deduce    from
  Lemma~\ref{lem:hp_2_cpct}   that   the   operator  $S_A$,   and   thus
  $S_A^\star$,  is power  compact.   We apply  the Krein-Rutman  theorem
  (Theorem~\ref{theo:rappel}~\ref{th:item:KR})  to   the  power  compact
  operator   $S_A^\star$,  which   has   a   positive  spectral   radius
  $\rho(S_A^\star)  = \rho(S_A)  >  0$:  there exists  $w  \in L^q_+$  a
  non-null  eigenfunction  of  $S_A^\star$  related  to  the  eigenvalue
  $\rho(S_A)$.                                                     Since
  $\rho(S_A) w  = S^\star_ A  w= \ind{A} S^\star(\ind{A} w)$,  we deduce
  that $\supp(w) \subset A$  and thus $\langle w,v \rangle>0$.
  
We now consider the following duality product:
\begin{equation}
  \label{eq:magicproduct}
  \scal{w,S_Av - \lambda v} = \scal{S_A^\star w,v} - \lambda\scal{w,v} = (\rho(S_A) - \lambda)\scal{w,v}.
\end{equation}
Under the hypothesis that $Sv = \lambda v$ on $A=\supp(v)$, the left
hand side of~\eqref{eq:magicproduct} is zero, and thus
$\rho(S_A)=\lambda$. This gives Point~\ref{lem:item:superequal}.

To prove Point~\ref{lem:item:superstrict}, since $S_Av-\lambda v$ is by hypothesis
positive on $A$, and $w\ind{A} = w$ is nonnegative and not identically
equal to zero, we get $\scal{w,S_Av - \lambda v}>0$. Then
use~\eqref{eq:magicproduct} again to get  that
 $\rho(S_A)>\lambda$.
\end{proof}

\subsection{Operators related to the SIS model}
\label{sec:T-and-others}
Let    $(\Omega,    \cg,    \mu)$    be   a    measured    space    with
finite non-zero  measure $\mu$, that is,  
$\mu(\Omega)\in  (0,   +\infty  )$.  Recall we simply write $L^p$ for
$L^p(\Omega, \cg, \mu)$. 
Observe              that
$L^\infty  \subset  L^p$ for $p \in (1,+\infty)$  and  that  the
identity map $\iota$  from $L^\infty $ to $L^p$ is a  bounded injection.
\begin{lem}[On compactness]
  \label{lem:cpct-2}
Let  $p \in  (1,+\infty)$.
Let  $T$ be an  operator  from $L^p$  to
  $L^\infty $. The  linear map  $\iota T$ is a  compact operator on
  $L^p$.
\end{lem}

\begin{proof}
  Since $L^p$  is reflexive, see~\cite[Corollary  IV.8.2]{dunford88}, we
  get by~\cite[Corollary  VI.4.3]{dunford88} that  $T$ and $\iota  $ are
  weakly  compact.  By~\cite[Theorem  5.85]{aliprantis06}, as  the space
  $L^\infty$ is  an AM-space, it satisfies  the Dunford-Pettis property,
  thus, by~\cite[Theorem 5.87]{aliprantis06}, the  operator $\iota T$ on
  $L^p$ is then compact. 
\end{proof}

Let $p\in (1, +\infty )$ and  $\gamma\in L^\infty $ such that $\gamma>0$
a.e..    Let  $T$   be  a   positive  operator   on  $L^\infty   $  such
that~\eqref{eq:T-in-Linfty}  holds   for  all  $f\in  L^\infty   $.   In
particular, as  $L^\infty $ is dense  in $L^p$ (for the  $L^p$-norm), we
can extend  $T$ by  density into  a bounded linear  map $\tp$  on $L^p$.
Recall that $M_{1/\gamma}$  denotes the multiplication by  $1/\gamma$.  Notice
that  under Assumption~\ref{assum:2},  see~\eqref{eq:T/gamma-in-Lp}, the
linear map $TM_{1/\gamma}$,  denoted by $\tgi$, can be seen  as a bounded
linear map  on $L^\infty$, and it  can also be  extended by density
into   a    bounded   linear   map    $\tgp$   on   $L^p$,    see   also
Fig.~\ref{fig:T_and_others}.

\begin{lem}[Properties of operators related to $T$]
  \label{lem:T}
  Let $(T, \gamma, \varphi)$ satisfy
Assumption~\ref{assum:1}. Then  $\tp$ is an operator  on
  $L^p$. If Assumption~\ref{assum:2} holds, then we have:
\begin{enumerate}[(a)]
\item\label{item:T}
  $\tp$ is a  compact operator on 
  $L^p$;
\item\label{item:ti}
   $T^2$  is a  compact operator on 
  $L^\infty $;
\item  \label{item:tg}
  $\tg{\gamma}$  is a  compact operator on 
  $L^p$;
\item \label{item:tgi} $\tgi$ is an  operator
  on   $L^\infty   $ and  $\tgi^2$ is compact;
 \item \label{item:spec}  $  \spec(T)=\spec(\tp)$
and $  \spec(\tgi)=\spec(\tgp)$. 

\end{enumerate}
\end{lem}

\begin{proof}
  Suppose      Assumption~\ref{assum:2}      holds.       We      deduce
  from~\eqref{eq:T/gamma-in-Lp} that there exists  a finite constant $C$
  such   that  $\norm{Tf}_\infty   \leq   C  \,   \norm{f}_p$  for   all
  $f\in L^\infty $. By density, we can extend $T$ into an operator $\tt$
  from  $L^p$ to  $L^\infty$. This  gives that  $\tp=\iota\,  \tt$  and
  Point~\ref{item:T} on   the
  compactness    property    of    $\tp$    is    a    consequence    of
  Lemma~\ref{lem:cpct-2}. As $T= \tt\, \iota$ and thus $T^2=  \tt \,
  \tp \, \iota$, we also get that $T^2$ is compact, that is Point~\ref{item:ti}. 
  Using~\eqref{eq:T/gamma-in-Lp}, we can also
 extend $TM_{1/\gamma}$ into an operator $\tt_{1/\gamma}$
  from  $L^p$ to  $L^\infty$. Arguing as above gives
  Points~\ref{item:tg} and~\ref{item:tgi}. 

\medskip

We    now    prove Point~\ref{item:spec}.      Two   complex    Banach    spaces
$(E,  \norm{\cdot})$   and  $(E',  \norm{\cdot}')$  are   compatible  if
$(E'',  \norm{\cdot}+ \norm{\cdot}')$,  with $E''=E\cap  E'$, is  a Banach
space, and $E''$ is dense in $E$ and in $E'$. Given two compatible
spaces $E$ and $E'$, two operators $A$  on $E$ and $A'$ on $E'$ are said
to  be   consistent  if  $A(E'')\subset  E''$,
$A'(E'')\subset E''$ and $Ax=A'x$ for all $x\in E''$. If furthermore $A$
and $A'$  are compact,  then \cite[Theorem~4.2.15]{davies07}  gives that
$\spec(A)=\spec(A')$.  The proof therein relies on the spectrum to be at
most countable and  with no accumulation points except  possibly $0$ and
that the spectral projections have   finite rank, see Theorems 5 and 6 p.579 in~\cite{dunford88}. Since this also holds
for power compact operators, the results can be extended to $A$ and $A'$
being power compact operators.

As $\mu$  is a finite measure,  the
spaces $L^p$ and  $L^\infty $ are pairwise compatible.   Notice also the
operators  $T$   and  $\tp$,   as  well  as   $\tgi$  and   $\tgp$,  are
consistent. Then the equalities follow.
\end{proof}

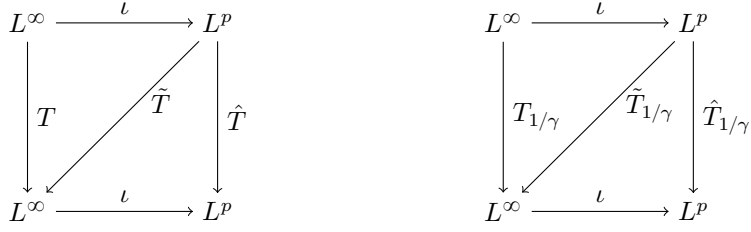
\begin{figure}
\centering\hspace{2cm}
\begin{subfigure}[b]{.4\textwidth} \hspace{0cm} 
\begin{tikzpicture}
\node[](4) at (0,-2.5) {$L^\infty$};
\node[](5) at (2.5,-2.5) {$L^p$};
\node[](7) at (0,-5) {$L^\infty$};
\node[](8) at (2.5,-5) {$L^p$};
\draw[->] (4) to node[midway, above]{$\iota$} (5);
\draw[->] (7) to node[midway, above]{$\iota$} (8);
\draw[->] (5) to node[midway, right]{$\tp$} (8);
\draw[->] (5) to node[near start, below]{$\tt$} (7);
\draw[->] (4) to node[midway, right]{$T$} (7);
\end{tikzpicture}
\label{fig:T1}
\end{subfigure}
\begin{subfigure}[b]{.4\textwidth}
\begin{tikzpicture}
\node[](4) at (0,-2.5) {$L^\infty$};
\node[](5) at (2.5,-2.5) {$L^p$};
\node[](7) at (0,-5) {$L^\infty$};
\node[](8) at (2.5,-5) {$L^p$};
\draw[->] (4) to node[midway, above]{$\iota$} (5);
\draw[->] (7) to node[midway, above]{$\iota$} (8);
\draw[->] (5) to node[midway, right]{$\tgp$} (8);
\draw[->] (5) to node[near start, below]{$\quad \tt _{1/\gamma}$} (7);
\draw[->] (4) to node[midway, right]{$\tgi$} (7);
\end{tikzpicture}
\label{fig:T2}
\end{subfigure}
\caption{Operators related to $T$ in the SIS model $(T, \gamma,
  \varphi)$.}
\label{fig:T_and_others}
\end{figure}

\section{Equilibria and restriction}
\label{sec:equilibre}

We      consider     the      SIS
model~\eqref{eq:intro_gen_SIS}-\eqref{eq:def-F}  on
$L^\infty =L^\infty  (\Omega, \cg,  \mu)$
with parameter $(T, \gamma, \varphi)$ such that
Assumption~\ref{assum:1} holds. In particular the measure $\mu$ is
finite and non-zero. 
We consider the following subset of $L^\infty $:
\[
  \Delta=\{f \in L^\infty \,\colon\,  \zero \leq f \leq \un\}.
\]
Recall  that $g$  is  an  equilibrium if  $g$  belongs  to $\Delta$  and
solves~\eqref{eq:gen_equilibre}, that is:
\[
  F(g)=0.
\]
 In particular,  the   function
$\zero$
 is  an equilibrium. We  say that $g^*\in  L^\infty $  is the
\emph{maximal  equilibrium} if  $g^*$ is  an equilibrium  and all  other
equilibrium $g\in L^\infty $ are such that $g\leq g^*$.

The  existence   result  of  the  semi-flow  and  the   maximal  equilibrium
follows~\cite[Propositions  2.7 and  2.15]{ddz-sis}  with slightly  more
general hypothesis  on $T$  and is  obtained similarly,  see a proof in 
Section~\ref{sec:exists_max_sol} for  completeness. We shall refer to
this section for notations and definitions/properties of the semi-flow. 

\begin{prop}[Existence of a global solution and of the maximal equilibrium]
  \label{prop:exists_max_sol}
  Let $(T, \gamma,  \varphi)$ be parameters of the  SIS model satisfying
  Assumption~\ref{assum:1}.  The following properties hold. 
\begin{enumerate}[(i)]
\item \label{prop:item:unique_sol}
 Equation~\eqref{eq:intro_gen_SIS}  in $L^\infty $ with
  initial condition $h\in \Delta$ has a unique global
  solution given by the semi-flow
  $\phi(\cdot, h)=(\phi(t, h))_{t\in \R_+}$. The semi-flow belongs to
  $\cc^1(\R_+)$. 
\item \label{prop:item:delta_inv}
  For all $t \geq 0$ and $h \in \Delta$, we have
  $\phi(t, h) \in \Delta$.
\item \label{prop:item:g*_lim}
  The sequence
  $(\phi(t, \un))_{t\in \R_+}$ is non-increasing and converges
  essentially to a limit,  $ g^*$, which is the maximal equilibrium:
  \[
    \limess_{t\rightarrow+\infty } \phi(t, \un)= g^*.
  \]
\end{enumerate}
\end{prop}
The    convergence     in
Point~\ref{prop:item:g*_lim} is not uniform  in general,
see Remarks~\ref{rem:R0<1} and~\ref{rem:non-unif-cv} and  
Example~\ref{ex:unif-cv-g=0}.

The proof of the monotonicity of the maximal equilibrium in  the parameters is
given in Section~\ref{sec:exists_max_sol} for consistency of the
arguments. 

\begin{lem}[Monotonicity of the maximal equilibrium]
  \label{lem:monot_eq_max_para}
  For $i=1,2$, let $(T_i, \gamma_i, \varphi_i)$ be parameters of the SIS
  model         satisfying        Assumption~\ref{assum:1} and denote
  $g^*_i$ the corresponding maximal equilibrium.  If 
   $T_1 \geq  T_2$,
 $\varphi_1  \geq   \varphi_2$ and $\gamma_1   \leq  \gamma_2$, 
  then we have $g^*_1 \geq g^*_2$.
\end{lem}

Let  $A$ be  a measurable  set.   Since $(T,  \gamma, \varphi)$  satisfy
Assumption~\ref{assum:1},  so  does   $(T_A,  \gamma,  \varphi)$,  where
$T_A= T_A =  M_A T M_A$ is the  projection of $T$ on $A$.   We shall now
focus on this restricted $(T_A, \gamma, \varphi)$-SIS model.  We set for
$u\in L^\infty$:
\begin{equation}\label{eq:def_F_A}
F_A(u) = \varphi(u) T_A(u) - \gamma u,
\end{equation}
and call  $g\in \Delta$  an equilibrium  of $A$  if $F_A(g)=0$. In this
case, notice that $\supp(g) \subset A$ a.e..  We also
denote   $\phi_A$   the   corresponding  semi-flow   and   $g^*_A$   the
corresponding maximal equilibrium of $A$ given by Proposition~\ref{prop:exists_max_sol}.

\begin{lem}[Maximal equilibria]\label{lem:comparaison_eq_max}
Let $(T, \gamma, \varphi)$  satisfy Assumption~\ref{assum:1}.
  Let $A\subset B$ a.e.\ be measurable sets. We have  $g^*_A\leq g^*_B$. 
\end{lem}
\begin{proof}
  Apply Lemma~\ref{lem:monot_eq_max_para} with  $T_1 = T_B\geq T_A=T_2$,
  $\varphi_1 = \varphi_2 = \varphi$ and $\gamma_1 = \gamma_2 = \gamma$.
\end{proof}

We now provide results on  equilibria and  semi-flows 
associated to $T$ and $T_A$.

\begin{lem}[Equilibrium and restriction]
  \label{cor:inv-equi}
  Let $(T,  \gamma, \varphi)$  satisfy Assumption~\ref{assum:1}. Let
  $A$ be a
  measurable set and $g \in \Delta$.
\begin{enumerate}[(i)]
\item \label{it:equi-equiA}
    If    $g$ is  an
  equilibrium and $\supp(g)\subset A  $  a.e., then $g$ is an equilibrium
  of  $A$.
\item \label{it:equiA-equi}
  If  $A$ is invariant and $g$ is an equilibrium
  of  $A$, then  $g$ is an equilibrium (and  $\supp(g)\subset A  $  a.e.).
\item \label{it:equi-equiAc}
  If $A$ is invariant and  $g$ is an equilibrium, then $\un_{A^c} g$
  is an equilibrium of $A^c$ and of $A^c \cap \supp(g)$.
\end{enumerate}
\end{lem}
\begin{proof}
  If $\supp(g)\subset A  $ and  $g$ is  an
  equilibrium, then we get that $g$ is an equilibrium of $A$ as:
  \[
    F_A(g)=\varphi(g) \un_A T (\un_A g) - \gamma g=
    \un_A \big( \varphi(g) T ( g) - \gamma  g\big)=0. 
  \]

\medskip

Let $A$ be an invariant set and $h\in \Delta$. Since $\un_{A^c}
T(\un_{A} h) = 0$, we deduce that:
\begin{equation}
  \label{eq:FAc-inv}
    F_{A^c} (\un_{A^c} h)= \un_{A^c} F(h).
  \end{equation}
  If furthermore $\supp(h)\subset A$, then we have:
\begin{equation}
  \label{eq:FA-inv}
    F_{A} (h)=  F(h) - \varphi(h) \un_{A^c} T (\un_A h)=  F(h).
  \end{equation}
  
  If $g$ is an equilibrium of  $A$, then we get $\supp(g)\subset A$ and,
  since $A$ is  invariant, we deduce from~\eqref{eq:FA-inv} that  $g$ is an
  equilibrium.    If  $g$   is  an   equilibrium,  we   deduce
  from~\eqref{eq:FAc-inv}  that  $\un_{A^c}  g$  is  an  equilibrium  of
  $A^c$. Then use~\ref{it:equi-equiA} with $T$ replaced by $T_{A^c}$ and $A$
  by $\supp(g)$ to  deduce that $\un_{A^c} g$ is also  an equilibrium of
  $A^c \cap \supp(g)$.
\end{proof}

\begin{lem}[Semi-flow and restriction]\label{lem:sf_rest}
Let $(T, \gamma, \varphi)$  satisfy Assumption~\ref{assum:1}.
Let $A$ be a measurable set and $h \in \Delta$.
The following properties hold:
\begin{enumerate}[(i)]
\item \label{lem:item:CV_sf_unA}
  $ \limess_{t\rightarrow+\infty } \phi_A(t, \un_A) = g^*_A$.
\item  \label{lem:item:rest_coinv}
  If  $A$  is invariant,  then we  have
  $\un_{A^c}\,  \phi(\cdot,h)  =  \phi_{A^c}\left(\cdot,  \un_{A^c}\,  h
  \right)$.
\item    \label{lem:item:rest_inv}
  If    $A$   is    invariant    and
  $\supp(h)      \subset      A$       a.e.,      then      we      have
  $\phi(\cdot,h) = \phi_A\left(\cdot, h \right)$.
\end{enumerate}
\end{lem}

\begin{proof}
  As $g^*_A$ is  an equilibrium of $A$, we get  $\supp(g^*_A) \subset A$
  and thus we have $g^*_A \leq  \un_A \leq \un$.  By the monotonicity of
  the semi-flow, see Lemma~\ref{prop:mono_sf}~\ref{prop:item:order_pre_flow},
  we                                                                have
  $g^*_A = \phi_A(t,g^*_A) \leq \phi_A(t,\un_A) \leq \phi_A(t, \un)$ for
  all             $t             \in             \R_+$.             Then
  Proposition~\ref{prop:exists_max_sol}~\ref{prop:item:g*_lim}     gives
  Point~\ref{lem:item:CV_sf_unA}.

\medskip
Let $A$ be invariant. 
For simplicity, we  write $\phi$ instead of $\phi(t, h)$ and $\phi'$ for
its derivative.
By definition of the semi-flow $\phi$, we deduce from~\eqref{eq:FAc-inv}
that:
\[
  \un_{A^c} \phi' = \un_{A^c} F(\phi)=  F_{A^c}
  (\un_{A^c} \phi)  . 
\]
The map $t \mapsto \un_{A^c} \phi(t, h)$ is thus a solution of
Equation~\eqref{eq:intro_gen_SIS}, with  $F$ replaced by $F_{A^c}$ and initial
condition $\un_{A^c}\,  h$. By the uniqueness of
the semi-flow,  we get 
$\un_{A^c} \phi(\cdot,h) = \phi_{A^c}\left(\cdot, \un_{A^c} h \right)$,
that is Point~\ref{lem:item:rest_coinv}.

We now assume that $\supp(h) \subset A$.
By Point~\ref{lem:item:rest_coinv},  we have $\supp( \phi)\subset A$.
We deduce from~\eqref{eq:FA-inv} that:
\[
  \phi'= F(\phi)=F_A(\phi). 
\]
Then, the same argument as above yields Point~\ref{lem:item:rest_inv}.
\end{proof}

\section{Characterization of  equilibria}\label{subsec:carac_eq}

In this  section, we assume that  Assumption~\ref{assum:2} holds for the
SIS model $(T, \gamma, \varphi)$. In particular  the map $\varphi$ restricted to $[0,1]$ is a decreasing bijection
  onto $[0,1]$. 
Recall the operators related to $T$ defined in
Section~\ref{sec:T-and-others} and their properties.

\subsection{Equilibria, supports and spectral radius}

For any $g \in \Delta$, let $L_g$ denote the compact  operator on
$L^p$ defined by:
\begin{equation}
  \label{eq:def_Lg}
  L_g = M_{\varphi(g)} \, \tg{\gamma}.
\end{equation}
This operator is associated to the linearization $M_{\varphi(g)} T - \gamma$ of the dynamics near $g$ in the same
way as $\tg{\gamma}$ is associated to the linearization $T -
\gamma$ of the dynamics  near $0$ (as $\varphi(0)=1$).
Notice that when $(T, \gamma, \varphi)$ satisfy
Assumption~\ref{assum:2}, then $(M_{\varphi(g)} T, \gamma, \varphi)$ also does.
It is immediate to check that for $g\in \Delta$:
\begin{equation}
   \label{eq:equiv-equi-Lg}
  g\text{ is an equilibrium}
\quad  \Longleftrightarrow\quad
L_g \, (\gamma g)=\gamma g.
\end{equation}

\begin{lem}[Equilibria as nonnegative eigenfunctions]
  \label{lem:eq_eigenf}
Let $(T, \gamma, \varphi)$  satisfy Assumption~\ref{assum:2}.
Let $g\in \Delta$ be an equilibrium. We have the following properties.
\begin{enumerate}[(i)]
	\item $\varphi(g) > 0$ a.e.. \label{lem:item:phi(g)>0}
	\item \label{lem:item:Tg_atoms}
          The operators $ \tp$, $\tg{\gamma}$ and $L_g$ have the
          same invariant sets,  irreducible sets, atoms and non-zero
          atoms. 
	\item The set $\supp(g)$ is invariant.
          \label{lem:item:supp_inv}
        \end{enumerate}
        If  $g\neq  \zero$,  then the  following  additional  properties
        hold. Let $h\in \Delta$ be an equilibrium. Set $A= {\supp(g)}$.
        \begin{enumerate}[resume*]
        \item   \label{lem:item:rho_TgA=1}
          $\rho(L_g)  \geq  \rho((L_g)_A) = 1$.
        \item\label{lem:item:compare_supports}
          If  $h\leq g$, then either
          $\rho((L_h)_A) = 1$ and $h=g$, or $\rho((L_h)_A) >1$
          and $\supp(h)\neq \supp(g)$ a.e..
        \item \label{lem:item:compare_supports2}
       $  h\leq g\,  \Longleftrightarrow\, \supp(h) \subset \supp(g)$ a.e..
\end{enumerate}
\end{lem}

As a consequence of Point~\ref{lem:item:compare_supports2}
we directly get the following corollary.
\begin{cor}[Equilibria and their support]
  \label{cor:supp-equi=}
Let $(T, \gamma, \varphi)$ that satisfy Assumption~\ref{assum:2}. Two equilibria with the same support are equal.
\end{cor}

\begin{proof}[Proof of Lemma~\ref{lem:eq_eigenf}]
Since $\varphi(g) T(g) = \gamma g$, $\varphi(g)$ does not vanish on
$\supp(g)$. On the complement set, $g=0$ so $\varphi(g)=1$ does not
vanish either.  Since $\varphi$ is nonnegative, we get
Point~\ref{lem:item:phi(g)>0}.

\medskip

Point~\ref{lem:item:Tg_atoms}   is   a   direct   consequence   of   the
characterization         of         invariant         sets         given
by~\eqref{eq:inv_supp_plus_petit}  as  $1/\gamma$ and  $\varphi(g)$  are
positive          by          Point~\ref{lem:item:phi(g)>0}          and
Thereom~\ref{theo:rappel}~\ref{th:item:dP}     on    non-zero     atoms.
Point~\ref{lem:item:supp_inv}      is      a     direct      consequence
of~\eqref{eq:equiv-equi-Lg}, \cite[Lemma 3.6]{dlz} (which state that the
support  of  the  eigenfunction   $\gamma  g$  is  $L_g$-invariant)  and
Point~\ref{lem:item:Tg_atoms}.

\medskip
Point~\ref{lem:item:rho_TgA=1} follows directly from Lemma~\ref{lem:rad>1}
with $S=L_g$, $v=\gamma g$ 
and $\lambda = 1$. 

The proof of Point~\ref{lem:item:compare_supports} follows similar
lines. Let $h\leq g$ be two equilibria with $h\neq \zero$.
The two eigenvalue equations written for $h$ and $g$ yield
$ (L_g-L_h) (\gamma g) + L_h(\gamma(g-h)) = \gamma (g-h)$, 
which we rewrite as:
\[
  L_h (\gamma(g-h)) = \gamma (g-h) +  (\varphi(h)-\varphi(g)) Tg.
\]
On  the  right  hand  side,  $\varphi(h)  -  \varphi(g)$  and  $Tg$  are
nonnegative,  and by  strict  monotonicity of  $\varphi$  they are  both
positive on  $B = \supp(g-h) \subset  \supp(g)$.  If $B$ is  empty, then
$g=h$ and  we are  back to  Point~\ref{lem:item:rho_TgA=1}.  If  not, we
apply  Lemma~\ref{lem:rad>1}~\ref{lem:item:superstrict}  to $S  =  L_h$,
$\lambda   =    1$   and   $v=\gamma (g-h)\neq   \zero$    which   is   non-negative:
$\rho((L_h)_B)> 1$.  If $B$ was a  subset of $\supp(h)$ this would imply
$\rho((L_h)_{\supp(h)})       >1$,       a      contradiction       with
Point~\ref{lem:item:rho_TgA=1}. So $B$ is not a subset of $\supp(h)$, or
in other words $\supp(h) \neq \supp(g)$.

\medskip            Finally            let           us            prove
Point~\ref{lem:item:compare_supports2}. Clearly  $h\leq g$  implies that
$\supp(h) \subset  \supp(g)$. To prove  the reverse implication,  let us
assume that $\supp(h)\subset\supp(g)$.  Since $F( \max(g,h)) \geq 0$, by
Lemmas~\ref{prop:mono_sf}~\ref{prop:item:mono_F}   and~\ref{lem:lim_eq},
the semi-flow starting from  $\max(g,h)$ is non-decreasing and converges
to    an   equilibrium    $\tilde{g}$,    which   therefore    satisfies
$\max(g,h)  \leq \tilde{g}$.   Since  $\supp(\max(g,h))  = \supp(g)$  is
invariant    by    Point~\ref{lem:item:supp_inv},   we    deduce    from
Lemma~\ref{lem:sf_rest}~\ref{lem:item:rest_inv}                     that
$\supp(\tilde         g)\subset        \supp(g)$         and        thus
$\supp(\tilde{g}) = \supp(g)$.  Since $g\leq \tilde{g}$, by the previous
point,   the  functions   $g$  and   $\tilde{g}$  must   be  equal,   so
$g = \max(g,h)$, or in other words $h\leq g$.
\end{proof}

\subsection{Maximum equilibria and critical vaccination}
\label{subsec:max_eq}
Recall the  notations of  Section~\ref{sec:T-and-others}.  Let $A$  be a
measurable set.  Notice the linear map $(\tg{\gamma})_A = \widehat{(T_A)}_{1/\gamma}$ is an operator
on  $L^p$.  Following~\cite{ddz-sis},  we  then  define the  \emph{basic
  reproduction number}  of $A$ as  the spectral radius of  this operator
$R_0(A)   =  \rho((\tg{\gamma})_A)$,   and   simply   write  $R_0$   for
$R_0(\Omega)$.  Notice  that, by~\eqref{eq:spec_rad_croissant},  the map
$A \mapsto R_0(A)$ is non-decreasing, that is, for any $A, B$ measurable
sets with $A \subset B$ a.e., we have $R_0(A) \leq R_0(B)$.

\medskip

The following result generalizes~\cite[Theorems 4.7 and 4.13]{ddz-sis},
and is proved similarly, 
see Section~\ref{sec:DDZ_cv_sf} for details. 
\begin{prop}\label{th:DDZ_cv_sf}
Let $(T, \gamma, \varphi)$  satisfy Assumption~\ref{assum:2}.
Then we have the following properties. 
\begin{enumerate}[(i)]
\item \label{th:item:R0<1}
  If  $R_0 \leq 1$, then we have $g^* = \zero$, and for all $h\in \Delta$:
  \[
    \limess_{t\rightarrow +\infty } \phi(t, h)=\zero.
  \]
  
\item \label{th:item:R0>1} If $R_0 >  1$, then the maximal equilibrium $g^*$  is non-null,
 (that  is, $\mu(\supp(g^*)) >  0$).
\item \label{th:item:R0>1-irr}
  If $R_0 >  1$ and  $T$ is
  quasi-irreducible, that  is, $T  = T_A$ with  $A$ an  irreducible set,
  then we  have $\supp(g^*) = A$  a.e.\ and $g^*$ is  the unique non-null
  equilibrium.
\item \label{th:item:R0>1-irr-lim}
  If $R_0 >  1$, $T$ is
  quasi-irreducible, that  is, $T  = T_A$ with  $A$ an  irreducible set,
  and $h\in \Delta$, then we
  have $\limess_{t\rightarrow +\infty } \phi(t, h)=\zero$ if
  $\supp(h) \cap A=\emptyset$ a.e.\ and:
  \[
    \limess_{t\rightarrow +\infty } \phi(t, h)=g^*
    \quad\text{if}\quad
    \mu(\supp(h) \cap A)>0. 
  \]
\end{enumerate}
\end{prop}

In  the next  remarks  and  examples, we  explore  the  uniformity of  the
convergence                in               Point~\ref{th:item:R0<1}
and Proposition~\ref{prop:exists_max_sol}~\ref{prop:item:g*_lim}.

\begin{rem}[Exponential rate of convergence to $\zero$ when $R_0<1$ and $\essinf \gamma > 0$]
  \label{rem:R0<1}
  Assume  $(T,  \gamma,  \varphi)$  satisfies  Assumption~\ref{assum:2},
  $R_0<1$ and  $\essinf  \gamma >  0$.  By  Lemma~\ref{lem:rad_bound} we
  also have  $s(T - \gamma)  < 0$.  Then, mimicking  the proof of
  \cite[Theorem~4.6]{ddz-sis} and  using that $1-\varphi\geq 0$,  we get
  that   for  all   $c  \in   (0,   -s(T  -   \gamma))$,  there   exists
  $\theta \in \R_+$  such that, for all  $h \in \Delta$, $t  \geq 0$, we
  have:
  \[
    \norm{\phi(t,h)}_\infty \leq \theta \norm{h}_\infty \, \expp{-ct}.
  \]
\end{rem}

\begin{rem}[Non uniform convergence when $R_0\leq 1$ and $\essinf \gamma=0$]
   \label{rem:non-unif-cv}
   We assume  that $R_0\leq 1$ and  $\essinf \gamma = 0$.   Consider the
   function  $v$  defined by  $v(t)  =  \exp(-t\gamma) \in  \Delta$  for
   $t \geq 0$.  As $v'(t) - F(v(t))=  - \varphi(v(t)) \, T v(t) \leq 0$,
   we        deduce        from        Lemma~\ref{lem:F_coop}        and
   Proposition~\ref{prop:comp_th} that  $\phi(t,\un) \geq v(t)$  for all
   $t\geq                 0$.                  We                 obtain
   $1\geq  \norm{\phi(t,\un)}_\infty  \geq   \norm{v(t)}_\infty  =1$  as
   $\essinf \gamma = 0$.  Thus the semi-flow $(\phi(t,\un))_{t\in \R_+}$
   does not converge to $g^*=\zero$ in $L^\infty $.

   Notice  the  same  conclusion  holds (with  the  same  arguments)  if
   $R_0\leq   1$   is   replaced   by   the   more   general   condition
   $\esssup g^*<1$.
 \end{rem}

\begin{ex}[A uniform convergence when $\essinf \gamma=0$]
  \label{ex:unif-cv-g=0}
  If $\esssup  g^*=1$ and $\essinf  \gamma = 0$  it is possible  for the
  semi-flow  $(\phi(t,\un))_{t\in   \R_+}$  to  converge  to   $g^*$  in
  $L^\infty  $.  Consider  the  particular case:  $\Omega=(0,  1]$  with
  $\mu=\nu  +  \delta_1$,  where  $\nu$  is  the  Lebesgue  measure  and
  $\delta_1$ the Dirac  mass at $1$; $Tf= f(1) \un$  for all $f\in L^p$;
  $\gamma(x)=x/2$  and $\varphi(r)=1-r$.  In this  case, $\{1\}$  is the
  only  atom,  and  $R_0(\{1\})=2$.  We get  $g^*(x)=1/(1+x)$  and  thus
  $\esssup g^*=1$. (Notice that $g^*(0+)=\esssup g^*$ and
  $\gamma(0+)=\essinf \gamma$.)
  Elementary     calculus     give      that,     for     $t\geq     0$,
  $\phi(t,\un)(1) = 1/(2-\expp{-t/2})$ and for $x\in (0, 1)$:
 \[
   \phi(t,\un)(x) =  \frac{2+ (x-1) \expp{-(x+1)t/2}
   }{x+1} \, \phi(t,\un)(1) , 
 \]
 so that  $\norm{\phi(t, \un) -  g^*}_\infty \leq \expp{-t/2}$.   So the
 semi-flow $(\phi(t,\un))_{t\in \R_+}$ converges  to $g^*$  in $L^\infty$.
\end{ex}

We  now  focus on  critical  vaccination.   Let $\varphi_0$  defined  by
$\varphi_0(r)=1   -r$    and   $\TT_k$   the   kernel    operator   form
Remark~\ref{rem:SIS-ddz}        with       $k$        and       $\gamma$
satisfying~\eqref{eq:hyp-k}   so  that   $(\TT_k,  \gamma,   \varphi_0)$
satisfies  Assumption~\ref{assum:2}. Let  $\eta  \in \Delta$  seen as  a
perfect vaccination strategy: the SIS model  $(\TT_k M_\eta, \gamma, \varphi_0)$
(which  indeed satisfies  Assumption~\ref{assum:2})  corresponds to  the
initial SIS model,  where for $x \in \Omega$, a
proportion $1-\eta(x)$ of the population is vaccinated and thus does not
spread the disease, see  \cite{ddz21targeted} and references therein. In
this setting vaccinating  the population amounts to  replace the measure
$\mu$ by $\eta \mu$.

Motivated by this example, we shall consider the effective
reproduction               number               defined
by:
\[
  R_e(\eta)  
  =  \rho( \tg{\gamma}  M_\eta) ,
\]
for $\eta\in \Delta$ (notice that $(TM_\eta, \gamma, \varphi)$ satisfies
Assumption~\ref{assum:2} and $\tg{\gamma}  M_\eta=\widehat{(TM_n)}_{1/\gamma}$). 
Following  \cite{ddz21targeted},  we  shall be  interested  in  critical
vaccination  $\eta$  for which  $R_e(\eta)  =  1$.   It is  observed  in
\cite{ddz_vacc_crit}        that         for        the        SIS model
$(\TT_k,     \gamma,    \varphi_0)$,     the    vaccination     strategy
$\eta=\varphi_0(g^*)$ is  critical.  We  generalize this result  (with a
shorter  proof  based  on  the  fact  that  $R_e(\varphi(g))=\rho(L_g)$,
see~\eqref{eq:def_Lg}) for more general  operators $T$ and functions 
$\varphi$.

\begin{theo}[Equilibria and critical vaccination]\label{prop:eq_crit_vacc}
  Let $(T, \gamma, \varphi)$ that satisfy Assumption~\ref{assum:2}. Let $h \in \Delta$ be an equilibrium.  Then we
  have $h = g^* \, \Longleftrightarrow \, R_e(\varphi(h)) \leq 1$.  If
  furthermore  $R_0 > 1$, then we have:
  \[
    h = g^* \, \Longleftrightarrow \, R_e(\varphi(h)) = 1.
  \]
\end{theo}

\begin{proof}
  First, remark that $\rho(L_g) = R_e(\varphi(g))$, where $L_g$ is
  defined by~\eqref{eq:def_Lg}.

If  $R_0 \leq 1$, then by Proposition~\ref{th:DDZ_cv_sf}, we have
  $g^* = \zero$. As $R_e(\varphi(0)) = R_0 \leq 1$ by
  Assumption~\ref{assum:2}, we directly get  the equivalence
  $h = g^* \, \Longleftrightarrow\,  R_e(\varphi(h)) \leq 1$.

  If    $R_0    >     1$,    then    we    have     $g^*\neq    \zero$    by
  Proposition~\ref{th:DDZ_cv_sf}.
According to Lemma~\ref{lem:eq_eigenf}~\ref{lem:item:rho_TgA=1}, if
$h\neq \zero$,  we have
$R_e(\varphi(h))=\rho(L_h)\geq 1$, and by~\eqref{eq:spec_rad_croissant} 
and~\ref{lem:item:compare_supports} that if $h\neq g^*$ then $\rho(L_h)
\geq  \rho\big((L_h)_A\big) >1$ with $A={\supp(g^*)}$. 

\medskip

To complete the  proof, that is, $R_e(\varphi(g^*))=1$,  we shall assume
that $\rho(L_h)>1$ and  show that $h\neq g^*$.   Informally the idea
is to follow the unstable direction  at the equilibrium $h$ to construct
a trajectory leading to another  equilibrium.  Since $L_h$ is a positive
compact    operator,     thanks    to    the     Krein-Rutman    theorem
(Theorem~\ref{theo:rappel}~\ref{thm:KR}), we can consider an eigenvector
$u\in L^p_+\priv{\zero}$ of $L_h$ related to $\rho(L_h)$.  Since the set
$A=\supp(h)$                is               invariant                by
Lemma~\ref{lem:eq_eigenf}~\ref{lem:item:supp_inv}, we have:
\begin{equation}
   \label{eq:Lhu}
 \rho(L_h) u  = L_h u = (L_h)_A u + L_h( \ind{A^c}u).
 \end{equation}
If $u\ind{A^c}$ was equal to $\zero$, $\rho(L_h) > 1$ would be an eigenvalue of $(L_h)_A$,  
contradicting Lemma~\ref{lem:eq_eigenf}~\ref{lem:item:rho_TgA=1}. As $A$
is invariant and $u\ind{A^c}\neq \zero$, multiplying~\eqref{eq:Lhu}
by $\ind{A^c}$ gives $\rho(L_h) u\ind{A^c}  =  (L_h)_{A^c} (
  u\ind{A^c})$,  showing  that  $u\ind{A^c}$ is an eigenvector of $(L_h)_{A^c}$, so
$\rho((L_h)_{A^c})= \rho(L_h) >1$. 

Since
$(L_h)_{A^c}  =  M_{\varphi(h)}  (\tg{\gamma})_{A^c} =
\widehat{(T_{A^c})}_{1/\gamma}$   (as    $h=0$   on   $A^c$ and $\varphi(0)=1$),    we   may   apply
Proposition~\ref{prop:rad_vpd}   with   $T=T_{A^c}$:  there   exists   a
$\lambda>0$    and     $w\in    L^\infty_+\priv{\zero}$     such    that
$T_{A^c} w - \gamma w = \lambda  w$.  Without loss of generality, we can
assume  that  $\norm{w}_\infty   $  is  small  enough   to  ensure  that
$ w  \in \Delta$  and, as $\varphi$  is continuous  with $\varphi(0)=1$,
that $\varphi(  \norm{w}_\infty) \geq  1- \delta$ with  $\delta>0$ small
enough so  that $\delta(\lambda+  \norm{\gamma}_\infty )  \leq \lambda$.
Note that $\supp(w) \subset A^c$.  Since $h+ w = h\ind{A} + w\ind{A^c}$,
and        $A=\supp(h)$        is        invariant,        we        get
$ \varphi(h+  w) Th = \varphi(h)  Th$.  Since $h$ is  an equilibrium, we
obtain that:
\begin{align*}
  F(h +  w)
= \varphi(h+ w) T(h+ w) - \gamma (h+ w) 
  &=  \varphi(h+ w)  Tw - \gamma w \\
  &\geq   (1-\delta) T_{A^c}w - \gamma w \\
  &\geq  \left( \lambda - \delta(\lambda + \norm{\gamma}_\infty
    ) \right) \, 
    w\\
&  \geq 0.
\end{align*}
By Lemmas~\ref{prop:mono_sf}~\ref{prop:item:mono_F}
and~\ref{lem:lim_eq}, this implies that the trajectory starting from $h+ w$
converges monotonously to an equilibrium $g$. Since $h\leq h+
w \leq g\leq g^*$ and $w\neq \zero$, we get that 
$h\neq g^*$ as claimed. 
\end{proof}

\subsection{Equilibria and antichains of atoms}
\label{sec:antichain}

We  now  focus  on  the   characterization  of  equilibria.   We  recall
from~\eqref{eq:F(C)} that the the future  of an antichain $\cc$ of atoms
(which      is     at      most     countable)      is     given      by
$ \cf(\cc) = \cf \left( \bigcup_{A \in \cc} A \right)= \bigcup_{A \in \cc} \cf
\left( A \right)$.
The  set of
\emph{supercritical} atoms:
\[
  \atomsc=\{A\in  \atom\, \colon\,  R_0(A)>1\}
\]
is finite by \cite[Lemma 6.5]{dlz}.  We  say an antichain $\cc$ of atoms
is supercritical if  all its elements are supercritical  atoms, that is,
$\cc\subset  \atomsc$.  We   denote  by  $\sac$  the   (finite)  set  of
supercritical antichains.   For a set  $A$, let $\cc_A$ denote the
(possibly empty)  supercritical antichain given by  the maximal elements
of $\{B \in \atomsc\, \colon\, B \subset A \text{ a.e.}\}$.  Notice that
when  $A$ is  admissible, we  get by~\eqref{eq:R0=max}  that $\cc_A$  is
non-empty if and only if $R_0(A)>1$.  For $h\in \Delta$, we simply write
$\cc_h$ for $\cc_{\supp(h)}$.

The following theorem generalizes the uniqueness result of
Proposition~\ref{th:DDZ_cv_sf} when the operator $T$ is not
necessarily quasi-irreducible. Recall that, by Lemma~\ref{lem:eq_eigenf}~\ref{lem:item:compare_supports2},
equilibria are characterized by their support.

\begin{theo}[Equilibria and supercritical antichains are in bijection]
  \label{th:bij_antichains_eq}
  Let $(T, \gamma, \varphi)$  satisfy Assumption~\ref{assum:2}.  The set
  of  the equilibria  and the  set  of supercritical  antichains are  in
  bijection through the equivalent relations:
\begin{equation}
   \label{eq:FC=g}
  \supp(g)= \cf(\cc)
  \quad  \Longleftrightarrow \quad
 \cc_g= \cc ,
\end{equation}
where $g\in \Delta$ is an  equilibrium and $\cc\in \sac$ a supercritical
antichain. Furthermore,  if $g\neq \zero$,  then the equilibrium  $g$ is
the maximal equilibrium of $\cf(\cc_g)$.
\end{theo}

We divide the proof in two lemmas. 

\begin{lem}[Support of an equilibrium and related supercritical antichain]
  \label{lem:supp=FC}
  If $g \in \Delta$ is an equilibrium, then  we have $\supp(g) =
  \cf\left(\cc_g\right)$ a.e.. 
  In particular if $g$ and $h$ are two equilibria, we get:
  \[
  \cc_g=\cc_h
    \Longleftrightarrow  g=h.
  \]
\end{lem}

\begin{proof}
  As the set $\supp(g)$ is invariant by
  Lemma~\ref{lem:eq_eigenf}~\ref{lem:item:supp_inv} and as every
  element of $\cc_g$ is included in $\supp(g)$, we have
  $\cf(\cc_g) \subset \supp(g)$.

By construction of $\cc_g$, every atom $B\subset \supp(g)$ with $R_0(B)>1$
is included in $\cf(\cc_g)$.
This implies by~\eqref{eq:R0=max} that, with $A=\supp(g) \cap \cf(\cc_g)^c$ an invariant (by Lemma~\ref{lem:eq_eigenf}~\ref{lem:item:supp_inv}) and thus admissible set:
\[
R_0(A) = \max_{ B\subset A, \, B\in \atom} R_0(B) \leq 1.
\]
Then Lemma~\ref{cor:inv-equi}~\ref{it:equi-equiAc}
gives that $\un_{\cf(\cc_g)^c} g$ is an equilibrium of $\cf(\cc_g)^c$ and
 of $A$. Then Proposition~\ref{th:DDZ_cv_sf}~\ref{th:item:R0<1} (with
 the SIS model
 $(T_A, \gamma, \varphi)$) implies
 that $\un_{\cf(\cc_g)^c}\,  g=\zero$, that is $\supp(g) \subset \cf(\cc_g)$. 
Thus, we get $\supp(g)=\cf(\cc_g) $. The second part of the lemma is
then a
direct consequence of Corollary~\ref{cor:supp-equi=} and
Lemma~\ref{lem:antichain_fut_eq}.  
\end{proof}

\begin{lem}\label{lem:ex_eq_support}
  For any supercritical antichain $\cc$, there exists an
  equilibrium $g\in \Delta $ such that  $\cc = \cc_g$. If $\cc$ is non
  empty, then  $g\neq \zero$ is the maximal equilibrium of $\cf(\cc)$.
\end{lem}

\begin{proof}
  If  $\cc$ is  empty, then  taking  the equilibrium  $g=\zero$, we  get
  $\cc=\cc_g$.   We   assume  now   that  $\cc$   is  not   empty.   Let
  $g$ be the maximal equilibrium on $\cf(\cc)$. It is also an
  equilibrium   by    Lemma~\ref{cor:inv-equi}~\ref{it:equiA-equi}   and
  $\supp(g)\subset \cf(\cc)$.  For  any $A\in\cc$, we have  $g^*_A \leq g$
  by Lemma~\ref{lem:comparaison_eq_max}, and $g^*_A$  is positive on $A$
  by Proposition~\ref{th:DDZ_cv_sf}~\ref{th:item:R0>1-irr}  since $A$ is
  a  supercritical atom  and thus  an irreducible  set with  $R_0(A)>1$.
  This implies that $A\subset \supp(g)$ and thus $\cf(A) \subset \supp(g)$
  as  $\supp(g)$   is  invariant.    Then  use~\eqref{eq:F(C)}   to  get
  $\cf(\cc)  \subset  \supp(g)$,  so   that  $\cf(\cc)=\supp(g)$,  and  thus
  $\cf(\cc)=\cf(\cc_g)$  by  Lemma~\ref{lem:supp=FC}.   The  two  antichains
  $\cc$  and  $\cc_g$  have  the  same future  and  are  thus  equal  by
  Lemma~\ref{lem:antichain_fut_eq}. The proof is then complete.
\end{proof}

Let $\eq\subset \Delta$ denote the set of equilibria of the SIS model
$(T, \gamma, \varphi)$. 
\begin{proof}[Proof of Theorem~\ref{th:bij_antichains_eq}]
The map $g \mapsto \cc_g$ from $\eq$ to $\sac$ is one-to-one by Lemma~\ref{lem:supp=FC} and onto by Lemma~\ref{lem:ex_eq_support}.
The equivalence given by~\eqref{eq:FC=g} is a direct consequence of
Lemmas~\ref{lem:antichain_fut_eq} and~\ref{lem:supp=FC}. Use the last
part of Lemma~\ref{lem:ex_eq_support} to get the last part of the theorem. 
\end{proof}

\subsection{Monatomicity and order relation via equilibria}

Let  $(T,  \gamma,   \varphi)$  that  satisfy  Assumption~\ref{assum:2}.
Consider the SIS model $(T, \lambda \gamma, \varphi)$ with recovery rate
$ \gamma$ multiplied  by a real parameter $\lambda>0$.   The
reproduction  number  of  a  measurable   set  $A$  for  this  model  is
$\rho((\tg{\lambda \gamma})_A)= \rho((\tg{ \gamma})_A\, 
M_{1/\lambda} )=R_0(A)/\lambda$. We  deduce from Theorem~\ref{th:bij_antichains_eq}  that the  number of
equilibria of the SIS model $(T, \lambda \gamma, \varphi)$ is decreasing
with $\lambda$.

\medskip

We say that the operator $T$ on $L^p$ for $p \in (1,
+\infty)$ is \emph{monatomic} if it has exactly one non-zero atom,
that is, $\card(\anz) = 1$. 
Monatomicity  is a natural extension of  (quasi-)irreducibility,
see~\cite[Remark~1.2]{dlz} and references therein. We complete the
characterization of monatomic  operator  $T$ given
in~\cite[Theorem~2]{dlz} using the number of equilibria of the SIS models $(T,
\lambda \gamma, \varphi)$. 

\begin{coro}[Criterium of monatomicity]
  \label{cor:monoatom}
Let $(T, \gamma, \varphi)$ satisfy Assumption~\ref{assum:2}.
The operator $T$ is monatomic if and only if the two following properties hold:
\begin{enumerate}[(i)]
\item \label{coro:item:card_leq_2}
  For  $\lambda > 0$, the
SIS model $(T, \lambda \gamma, \varphi)$  has at most one
  non-null equilibrium. 
\item  \label{coro:item:card=2}
  There   exists  $\lambda   >   0$   such  that   the   SIS model
  $(T,     \lambda     \gamma,      \varphi)$     has     a     non-null
  equilibrium. 
\end{enumerate}
\end{coro}

\begin{proof}
 By  Theorem~\ref{th:bij_antichains_eq}, we  deduce  that if  $\cc$ is  a
finite    antichain     of    non-zero    atoms    then,     for    all
$\lambda\in (0, \min   _{A\in  \cc}   R_0(A))$  there   exists  an   equilibrium
$g^\lambda$  for the  model  $(T, \lambda  \gamma,  \varphi)$ such  that
$\cf(\cc)=\supp(g^\lambda)$.

Then, Point~\ref{coro:item:card_leq_2}  means that the number  of finite
antichains    of    non-zero    atoms    is    at    most    two,    and
Point~\ref{coro:item:card=2} that it is at  least two: so the two points
are equivalent to the number of  antichains being exactly two (one being
empty), that is $T$ is monatomic.
\end{proof}

\section{Convergence and attraction domains}

In this  section, we  are interested  in the  behavior of  the semi-flow
$(\phi(t,  h))_{t\in \R_+}$ of  Equation~\eqref{eq:intro_gen_SIS} when  $t$
goes to infinity for  an initial condition $h \in \Delta$.   If $T$ is a
quasi-irreducible   kernel   positive   operator,  then   according   to
Proposition~\ref{th:DDZ_cv_sf}~\ref{th:item:R0>1-irr-lim},    see   also
\cite[Theorem 4.13]{ddz-sis} when $T$ is an irreducible kernel operator,
the  semi-flow  $(\phi(t,  h))_{t\in \R_+}$  converges   essentially     to
$g^*$ if $\mu(\supp(h) \cap \supp(g^*)) > 0$ and $\zero$ otherwise.   We   generalize  this  result  to  general
operators, see Section~\ref{sec:sf_cv} for a proof.

\begin{theo}[Convergence to an equilibrium]\label{th:cv_eq_max}
  Let $(T, \gamma, \varphi)$ satisfy Assumption~\ref{assum:2}.  The semi-flow
  $(\phi(t,  h))_{t\in \R_+}$ with initial condition $h\in \Delta$
  converges essentialy to a limit, say $g\in \Delta$; and $g$ is an equilibrium and
  more precisely the maximal
  equilibrium of the set $\cf(\supp(h))$:
  \[
\limess_{t\rightarrow \infty } 
\phi(t, h) = g
\quad\text{and}\quad
\cc_g=\cc_{\cf(\supp(h))}
.
\]
\end{theo}

We derive  directly the  next corollary,  where the  maximal equilibrium
$g^*$ is possibly equal to $\zero$.

\begin{cor}[Attraction domain of the maximum equilibrium]
   \label{cor:cv-g*}
   Let  $(T, \gamma,  \varphi)$  satisfy Assumption~\ref{assum:2}.   The
   semi-flow with  initial condition $h\in \Delta$  converges  essentially
   to      the     maximal      equilibrium     $g^*$,      that     is,
   $\limess_{t\rightarrow  \infty  }  \phi(t,  h)  =  g^*$,  if and only
   if  $\cf(\supp(h))$ contains all the supercritical atoms.
\end{cor}

In                 particular,                we                 recover
Proposition~\ref{th:DDZ_cv_sf} as $R_0\leq  1$ means
there  is no  supercritical atom, and $R_0>1$ and $T$ quasi-irreducible
means there is only one supercritical atom.   In  the previous  corollary, it  may however happen that  none  of   the  supercritical  atoms  is included  in
$\supp(h)$, see the next example.

\begin{ex}
  Let $\Omega  = \{a,b\}$  with the counting  measure, and  consider the
  SIS model $(T, \gamma,  \varphi)$ with $T$ identified  with the matrix
  $\begin{pmatrix}
    1 & 0 \\
    \star & 2
  \end{pmatrix}$  (with $\star  >  0$), $\gamma  =  \un$, $\varphi(r)  =
  1-r$. Notice that  $\{a\}$ and $\{b\}$ are non-zero  atoms, the former
  being  critical   with  $\cf(\{a\})=\Omega$   and  the   latter  being
  supercritical and invariant.  Thus,  there exists only two equilibria:
  $\zero=(0,0)$  and  $g^*=(0,1/2)$.  If  $h\neq  \zero$,  then we  have
  $\lim_{t\rightarrow \infty  } \phi(t,  h) =  g^*$.  For  $h=(0,1)$, we
  have   $\supp(h)=\supp(g^*)$,   but   for    $h   =   (1,0)$   we   have
  $\supp(h) \cap \supp(g^*) = \emptyset$.
\end{ex}

\begin{prop}[Uniform convergence to an equilibrium]
  \label{prop:cv-unif}
 Let  $(T, \gamma,  \varphi)$  satisfy 
  Assumption~\ref{assum:2}.  For
 $h \in \Delta$ and $g$ the maximal equilibrium of $\cf(\supp(h))$, we have   that:
\begin{equation}
   \label{eq:cv-unif-g>0}
\lim_{t\rightarrow \infty }     \norm{\big(\phi(t,h)-g\big)\, \gamma }_\infty =0. 
  \end{equation}  
\end{prop}

In particular, when $\essinf \gamma > 0$, the convergence given by Theorem~\ref{th:cv_eq_max} is uniform.
Notice we have a stronger result if furthermore  $R_0<1$ (and thus
$g=\zero$), see 
Remark~\ref{rem:R0<1}. 

\begin{proof}
  We start with a preliminary result.  Set
  $M$ the  norm of the operator  $\tt $ from $L^p$  to $L^\infty$, which
  coincide with $T$ on $L^\infty  $, see the proof of Lemma~\ref{lem:T};
  it  is  finite  by  Assumption~\ref{assum:2}.  Let $m>0$. Let  $h_1\geq  h_2$  be
  elements     of    $\Delta$,     and    for     $t\in    \R_+$     set
  $f(t) = \phi(t, h_1) - \phi(t, h_2)$. We claim that:
\begin{equation}
  \label{eq:fi<fp}
  \limsup_{t\rightarrow \infty }  \norm{f(t)\ind{\{\gamma\geq m\}} } _\infty  \leq
  \frac{M}{m}\, 
  \limsup_{t\rightarrow \infty }  \norm{f(t)}_p. 
\end{equation}
Indeed, by monotonicity of the semi-flow (see Lemma~\ref{prop:mono_sf}),
we have $\phi(t, h_1) \geq  \phi(t, h_2)$ and thus $f(t) \geq \zero$. We
also have, as $\varphi$ is decreasing  on $[0,1]$ and $T$ is positive
that for $t\in \R_+$:
\begin{align*}
  f'(t)
  &=  \varphi(\phi(t, h_1)) T\phi(t, h_1) - \varphi(\phi(t, h_2))
    T\phi(t, h_2) - \gamma f(t)\\
  &\leq  \varphi(\phi(t, h_1)) Tf(t) - \gamma f(t)\\
  &\leq  Tf(t) - \gamma f(t).    
\end{align*}
On $\{\gamma\geq  m\}$, we get for $v(t)=\expp{mt} f(t)$ that  for all $t \geq 0$:
\begin{equation}\label{proof:eq:v'_leq_Mv}
v'(t) \leq (m - \gamma)  v(t)+
T v(t) \leq \norm{\tt \, \iota  v(t)}_\infty  \leq M \norm{v(t)}_p .
\end{equation}
By~\eqref{eq:Bocner-norm}   and~\eqref{eq:calculus}   on   the   Bochner
integral,                 we                 deduce                 that
$v(t)  \leq v(0)  +  M \int_0^t  \norm{v(s)}_p\, \rd  s$ on $\{\gamma\geq  m\}$.  Since $f$  is
nonnegative, we get:
\[
  \norm{f(t)\ind{\{\gamma\geq m\}}}_\infty  \leq  \expp{-mt} \norm{h_1 -h_2}_\infty
  + M \int_0^t \expp{-m(t-s)} \norm{f(s)}_p\, \rd t
  \leq  2 \expp{-mt} 
  + M \int_0^t \expp{-m s} \norm{f(t-s)}_p\, \rd t. 
\]
This gives~\eqref{eq:fi<fp}.

\medskip

Let   $h  \in   \Delta$  and   $g$   be  the   maximal  equilibrium   of
$A=\cf(\supp(h))$.       By     Lemma~\ref{lem:sf_pos},      we     have
$\supp(\phi(1,h))   =  A$.    Let   $h_1  =   \max(\phi(1,h),  g)$   and
$h_2  =   \min(\phi(1,h),  g)$.    We  thus   have  $h_1\leq   h_2$  and
$\supp(g)=\supp(h_2)\subset \supp(h_1)=A$. By Theorem~\ref{th:cv_eq_max}
(and using  that $g$  is the  maximal equilibrium  on the  invariant set
$\cf(\supp(h_i))$      for     $i=1,      2$),      we     get      that
$\limess_{t\rightarrow  \infty }  \phi(t,  h_i)=g$ for  $i=1, 2$.   With
$f(t)  =  \phi(t,  h_1)  -  \phi(t,  h_2)$,  we  get  by  the  dominated
convergence theorem that  $\lim_{t\rightarrow \infty } \norm{f(t)}_p=0$,
and                       by~\eqref{eq:fi<fp}                       that
$\lim_{t\rightarrow   \infty   }   \norm{f(t)\ind{\{\gamma\geq
    m\}}}_\infty   =0$ for all $m>0$ and thus
$\lim_{t\rightarrow   \infty   }   \norm{ f(t)\, \gamma}_\infty   =0$ as
$\norm{f}_\infty \leq  1$.   Use   the
monotonicity        of        the         semi-flow        to        get
$\phi(t, h_1) \geq \phi(t+1, h) \geq \phi(t, h_2)$ and deduce
that~\eqref{eq:cv-unif-g>0} holds. 
\end{proof}

\section{SIS model with an external disease reservoir}

In  this section,  we  consider  the infinite-dimensional  inhomogeneous
SIS  model  with an  external  disease  reservoir, called  SIS$\kappa$
model,   presented   in  Section~\ref{sec:intro-SISk}.    The   function
$u=(u(t,x))_{t\in \R_+, x\in \Omega}$,  where $u(t,x)$ is the proportion
of  infected  population  among  the population  with  feature  $x$,  is
solution in $L^\infty $ of the ODE:
\begin{equation}\label{eq:gen_SIS_immigration}
\left\{\begin{array}{l}
u' = F_\kappa(u), \\ 
u(0) = h,
\end{array}\right.
\end{equation}
with initial condition $h\in L^\infty $ and:
\begin{equation}\label{eq:intuition_eq_immi}
F_\kappa(u)=  \varphi(u) (Tu+\kappa) - \gamma u,
\end{equation}
where    $\varphi$   is    a   continuous    function   on    $\R$   and
$\kappa\in    L^\infty    _+$. 
To   study    solutions   of~\eqref{eq:gen_SIS_immigration}    and   the
corresponding equilibria,  that is  functions $g  \in \Delta$  such that
$F_\kappa(g)    =    0$,    we    shall    use    the    formalism    of
Section~\ref{sec:equilibre} by adding  a sub-population corresponding to
the reservoir  with type $\r$.  Notice the case  $\varphi([0, 1])=\{0\}$
(which is  not possible under Assumption~\ref{assum:2})  is trivial, and
thus   we  shall   assume  there   exists  $a\in   (0,  1)$   such  that
$\varphi(a)>0$.

\medskip

We  set  $\Omega_\r=\Omega  \sqcup  \{\r\}$ (assuming  without  loss  of
generality that $\r\not\in \Omega$),  $\cg_\r = \sigma(\cg, \{\r\})$ and
$\mu_\r$ a measure  on $(\Omega_\r, \cg_\r)$ which coincides  with $\mu $
on $\Omega$  and with positive  finite weight  on $\r$.  For  a function
$f_\r$ defined on  $\Omega_\r$, we simply write $f$  for its restriction
to $\Omega$ (and  similarly, for a function $f$ defined  on $\Omega$, we
write $f_\r$ for a function  defined on $\Omega_\r$ which coincides with
$f$ on $\Omega$,  the value of $f_\r$ on $\r$  being given when needed).
We  simply write  $L^p_\r$  for  $L^p(\Omega_\r, \cg_\r,  \mu_\r)$,
where $p\in [1,  +\infty ]$.  We define the positive  operator $T_\r$ on
$L^\infty_\r$, as an  extension of $T$ on $\Omega_\r$,
by:
\[
  T_\r f_\r(x)= \ind{\Omega}(x) \, Tf(x) 
  + f_\r(\r)\,  \alpha_\r(x),
\]
with  $\alpha_\r\in (L^\infty  _\r) _+$  such that  $\alpha=\kappa/a$ on
$\Omega$,  with  $a\in  (0,  1)$, and  $\alpha_\r(\r)=b>0$. For $p \in (1,+\infty)$, we  define
similarly  the  operator  $\tp_\r$  on $L^p_\r$ based on the operator $\tp$ on $L^p$, see Lemma~\ref{lem:T}.   We  also  define  the
function  $\gamma_\r$  (which coincides  with  $\gamma$  on $\Omega$  by
definition)   such  that   $\gamma_\r(\r)=   b\varphi   (a) $ is assumed
to be positive (that is, $\varphi(a)>0$).    Let
$\Delta_\r=\{f_\r\in  (L^\infty _\r)_+\,  \colon\, 1-f_\r  \in (L^\infty
_\r)_+\}$ be the analogue of  $\Delta$ for $\Omega_\r$. It is elementary
to check the following result.
\begin{prop}[Solution to the SIS$\kappa$ model]
  \label{prop:sol-reservoir}
  Let  $(T, \gamma,  \varphi)$ satisfy  Assumption~\ref{assum:1}. Assume
  furthermore there exists  $a\in (0, 1)$ such  that $\varphi(a)>0$, and
  let       $\kappa\in       L_+^\infty       $.        A       function
  $(u(t,    x))_{t\in    \R_+,    x\in   \Omega}$    is    a    solution
  to~\eqref{eq:gen_SIS_immigration}  related  to the  SIS$\kappa$  model
  with  initial condition  $h\in \Delta$  if  and only  if the  function
  $(u_\r(t, x))_{t\in \R_+, x\in  \Omega_\r}$, where $u_\r(t, \r)=a$ for
  all $t\in \R_+$, is  a solution to~\eqref{eq:intro_gen_SIS} related to
  the  SIS  model  with  parameter  $(T_\r,  \gamma_\r,  \varphi)$  on
  $\Omega_\r$ and  with initial condition $h_\r\in  \Delta_\r$ such that
  $h_\r(\r)=a$.
\end{prop}

 We shall consider the supercritical atoms out of the individuals
infected by the reservoir:
\[
  \atomsc_\r=\{A \in \atomsc\, \colon\, A\cap
  \cf(\supp(\kappa))=\emptyset \text{ a.e.}\}.
\]
Based on  Theorems~\ref{th:bij_antichains_eq} and~\ref{th:cv_eq_max} for
the $(T_\r, \gamma_\r, \varphi)$ SIS  model we can give a representation
of the  equilibria of SIS$\kappa$  model, that  is, of the  solutions to
$F_\kappa(g)=0$  in  $\Delta$,  prove that the equilibria are
characterized by their support, and explicit  their  attraction  domain.   For  a
function  $h\in  L^\infty  _+$,  we   shall  denote  $\cc_{\r,  h}$  the
  antichain    given   by   the   maximal    elements   of
$\{B   \in   \atomsc_\r\,   \colon\,   B   \subset   \cf(\supp(h))
\text{   a.e.}\}$. Notice that Assumption~\ref{assum:2} implies that
$\varphi$ is positive on $[0, 1)$. 

\begin{cor}[Equilibria of the SIS$\kappa$ model]
  \label{cor:eq-reservoir}
  Let $(T,  \gamma, \varphi)$  satisfy Assumption~\ref{assum:2}
  and $\kappa\in L_+^\infty $.
  The set of equilibria and the set of antichains in $\atomsc_\r$ are in
  bijection through the equivalent relations:
\[
  \supp(g) = \cf(\cc)\cup \cf(\supp(\kappa)) \quad\text{a.e.}
  \quad  \Longleftrightarrow \quad
  \cc_{\r,g} =\cc,
\]
where $\cc\subset \atomsc_\r$ is an antichain and $g\in \Delta$ an
equilibrium, that is, $F_\kappa(g)=0$. 

Furthermore,   the  semi-flow   $(\phi(t,   h))_{t\in  \R_+}$   solution
of~\eqref{eq:gen_SIS_immigration} with  initial condition  $h\in \Delta$
is well  defined and it converges  a.e.\ to a limit,  say $g\in \Delta$;
and $g$ is an equilibrium and more precisely:
  \[
\limess_{t\rightarrow \infty } 
\phi(t, h) = g
\quad\text{and}\quad
\cc_{\r,g}=\cc_{\r, h}
.
\]
\end{cor}

\begin{rem}
  \label{rem:reservoir}
We deduce the following  properties under the hypothesis of
Corollary~\ref{cor:eq-reservoir}. 
  \begin{enumerate}
   \item Notice that   $\atomsc_\r$ is empty if and only if 
      there exists a unique equilibrium $g\in \Delta$ for the
      SIS$\kappa$ model. In this case,
     we have $\supp(g) = \cf(\supp(\kappa))$ a.e.\
     and  $\limess_{t\rightarrow \infty } 
     \phi(t, h) = g$ for all $h\in \Delta$.
   \item If $\essinf \kappa>0$, then $\atomsc_\r$ is empty.
     \item If $g$ is an equilibrium  for the
      SIS$\kappa$ model, then $g$ is positive on $\cf(\supp(\kappa))$
      and thus on $\supp(\kappa)$.  
  \end{enumerate}
\end{rem}

\begin{proof}[Proof of Corollary~\ref{cor:eq-reservoir}]
If $A\in  \cg$ is  invariant for  $\tp$, then  $A$ seen  as an  element of
$\cg_\r$  is  also  invariant  for $\tp_\r$.  Furthermore,  the reservoir
$\{\r\}$ is an atom of $\mu_\r$ and, as $\Omega$ is $\tp_\r$-invariant, we
get that $\{\r\}$ is a  $\tp_\r$-atom.  We deduce that
a  set  $B\in   \cg_\r$  is  admissible  for  $\tp_\r$  if   and  only  if
$B \cap  \Omega$ is  admissible for  $\tp$.
In particular a set $A$ is an atom of $\tp_\r$ if and only if either
$A=\{\r\}$ or $A \subset \Omega$ and $A$ is an atom of $\tp$.
We  denote by  $\cf_\r(A)$ the
future of a set $A\subset \Omega_\r$ with respect to $\tp_\r$. Notice that
$\cf_\r(A)=\cf(A)$ for any $ A\subset \Omega$ and that the future of the
reservoir is:
\begin{equation}
   \label{eq:Fr(r)}
  \cf_\r(\{\r\})=\{\r\} \cup \cf(\supp(\kappa)).
\end{equation}

Recall the  basic reproduction number $R_0(A)$
of a measurable set for the SIS model $(T, \gamma, \varphi)$ 
given in Section~\ref{subsec:max_eq}. We simply write $R_\r(A)$ when
considering the  basic reproduction number of $A \in \cg_\r$ for the 
 SIS model $(T_\r, \gamma_\r, \varphi)$. Since $(T_\r)_A=T_A$ for
 $A\subset \Omega$ measurable, we deduce that 
$R_\r(A)=R_0(A)$ for any $ A\subset \Omega$ measurable.
We also have:
\[
  R_\r(\{\r\})=\frac{\alpha_\r(\r)}{\gamma_\r(\r)}=\frac{1}{\varphi(a)}\cdot
\]
Note that under  Assumption~\ref{assum:2}, we have $\varphi((0, 1))=(0,
1)$ and thus the atom $\{\r\}$ is super-critical for the  $(T_\r,
\gamma_\r, \varphi)$ SIS model.

\medskip

We  deduce from  Proposition~\ref{prop:sol-reservoir}.  that  a function
$g\in \Delta$ is an equilibrium for the SIS$\kappa$ model if and only if
$g_\r\in \Delta_\r$,  such that $g_\r(\r)>0$  is an equilibrium  for the
$(T_\r,  \gamma_\r, \varphi)$  SIS  model on  $\Omega_\r$.  Notice  that
necessarily $g_\r(\r)=a$.
The equilibria of the  $(T_\r,
\gamma_\r, \varphi)$ SIS model 
whose support contains the reservoir $\{\r\}$ are according to
Theorems~\ref{th:bij_antichains_eq} in bijection will all the
supercritical antichains containing the atom $\{\r\}$. Those
supercritical  antichains
are exactly the antichains of $\atomsc_\r$ with the atom $\{\r\}$ added
to them.
Then, use~\eqref{eq:Fr(r)} and Theorems~\ref{th:bij_antichains_eq}
and~\ref{th:cv_eq_max} to conclude.
\end{proof}

\bibliographystyle{abbrv}
\bibliography{biblio}

\section{A short reminder on integration, derivation and ODE in  Banach
  space}
\label{sec:bochner_int}

This section is devoted to the  definition and properties of the Bochner
integral, the  differentiation and differential equations in
Banach spaces. We consider  $(X, \norm{\cdot})$ a real Banach space.

\subsection{Integration in Banach spaces}

We give a  short summary on the Bochner integral,
and  refer to  \cite{arendt2013vector} and  \cite{seifert22} for  a more
detailed presentation.
We  consider the  Borel  $\sigma$-field  on $\R$  and  write
$\nu(\rd  t)=\rd  t$  for  the Lebesgue  measure. Let $I$  be an interval of  $\R$.  A function
$f\, \colon\, I \mapsto X$   is \emph{simple}  if $f= \sum_{k=1}^n  a_k \ind{A_k}$
  where $n\in \N$,  the $a_k$'s belong to $X$ and  the $A_k$'s are Borel
  subsets of  $I$ with finite  Lebesgue measure.  We define  its Bochner
  integral as:
\[
  \int _I f\, \rd \nu
  = \sum_{k=1}^n a_k \, \nu(A_k).
\]
Notice  the  integral  belongs  to  $X$  and  does  not  depend  on  the
representation    of   the    simple    function    $f$.   A    function
$f\,  \colon\,  I  \mapsto  X$  is  Bochner  measurable  (simply  called
measurable  in~\cite{arendt2013vector}) if  there exists  a sequence  of
simple functions $(f_n)_{n\in  \N}$, with $f_n\, \colon\,  I \mapsto X$,
such that $\nu$-a.e.\ $\lim_{n\rightarrow\infty } f_n =f$  (that is,
$\lim_{n  \rightarrow \infty  } \norm{f(t)  -f_n(t)}=0$ $\rd  t$-a.e.\ on
$  I$); it  is  furthermore  Bochner integrable  if  one  can find  such
approximating      sequence      $(f_n)_{n\in     \N}$      so      that
$\lim_{n  \rightarrow \infty  }  \int  _I \norm{f(t)  -  f_n(t)} \,  \rd
t=0$. In this case the \emph{Bochner integral} of $f$ is defined as:
\[
  \int _I f\, \rd \nu
  = \lim_{n\rightarrow \infty }  \int _I f_n\, \rd \nu
,
 \]
 where  the limit  holds in  the  Banach space.  The Bochner  integrable
 $ \int _I  f\, \rd \nu$ does not depend  on the approximating sequence
 $(f_n)_{n\in \N}$; we shall also denote it by $\int_I f(t)\, \rd t$. 
Thanks to~\cite[Corollary 1.1.2]{arendt2013vector} $X$-valued continuous
function are Bochner measurable and    a.e.\ limits of 
 Bochner measurable function  are Bochner measurable. 
 According to~\cite[Corollary
 1.1.2]{arendt2013vector}, a function $f\, \colon\, I \mapsto X$ is Bochner integrable if and only
 if it is Bochner measurable and $\norm{f} \, \colon\, I \mapsto \R_+$
 is integrable; in this case we have:
\begin{equation}
   \label{eq:Bocner-norm}
   \norm{\int _I f\, \rd \nu}\leq  
   \int _I \norm{f}\, \rd \nu.
 \end{equation} 
When $\bar I=[a,b]$ with $-\infty \leq a<b\leq +\infty$  and $f$ is
Bochner integrable on $I$, we simply 
write the Bochner integral as:
\[
  \int _I f \, \rd \nu=\int_a^b f \rd \nu.
\]
 The Bochner integral enjoy many properties as the usual Lebesgue
 integral, as such we shall use the dominated convergence, see
 \cite[Theorem 1.1.8]{arendt2013vector}, which we recall.

\begin{theo}[Dominated convergence theorem]\label{th:TCD}
  Let $I$ be a non-empty interval of  $\R$.  Let $(f_n)_{n \in \N}$ be a
  sequence  of   Bochner-integrable  functions  defined  on   $I$  which
  converges   $\nu$-a.e.\ to $f\,  \colon\, I \rightarrow  X$. Assume
  there        exists        a       Lebesgue-integrable        function
  $g \, \colon\, I \rightarrow \R_+$ such that $\norm{f_n} \leq g$ $\nu$-a.e.\
  for all $n \in \N$.  Then the function $f$ is Bochner-integrable and:
\[
  \int_I f \, \rd \nu = \lim_{n\rightarrow \infty } \int_I f_n\, \rd \nu
  \quad\text{in $X$.}
\]
\end{theo}

\subsection{Differential equations in Banach spaces}
\label{sec:differential}
We now consider the derivation of  functions on a Banach space.  Let $I$
be  an interval  of $\R$  with non-empty  interior.  We  say that  a
function 
$f  \,  \colon\,  I  \mapsto  X$ is  \emph{differentiable}  at a  point
$t   \in   I$    if   the   following   limit,    $f'(t)$,   exists   in
$(X, \norm{\cdot})$:
\[
 f'(t)= \lim_{\substack{s \rightarrow 0\\t+s\in I}} \frac{f(t+s) - f(t)}{s}\cdot
\]
Notice that  if $f$ is differentiable  at $t$, then it  is continuous at
$t$.  We say  that $f$ is differentiable on $I$  if it is differentiable
at  any point  of  $I$, and  that  $f$  belongs to  $\cc^1(I)$  if it  is
differentiable  on $I$  and  $f'$ is  continuous on  $I$.
We   have   the  following   fundamental   theorem   of  calculus,   see
\cite[Proposition    1.2.2]{arendt2013vector}     and    \cite[Corollary
3.1.7]{seifert22}.  Assume  $I=[a,b]$ with  $-\infty <a<b<+\infty  $ and
that $f\, \colon\, I \rightarrow X$ belongs to $\mathcal{C}^1(I)$, then $f'$
is Bochner-integrable on $I$ and we have:
\begin{equation}
   \label{eq:calculus}
f(b) -f(a)= \int_a^b f'\, \rd \nu.
\end{equation}

\medskip

We now recall  some results on differential equations  in Banach spaces.
Let  $F\, \colon\, X \rightarrow X$ be locally-Lipschitz, that is,  for
all $x \in X$, there exists $\eta > 0$ and $ C$ finite such that for all
$y \in X$, we have:
\[
  \norm{x-y} \leq \eta\quad  \Longrightarrow \quad  \norm{F(x) - F(y)} \leq C
  \norm{x-y}.
\]
The  Picard-Lindelöf  theorem, see~\cite[Corollaries~IV~1.6-8]{lang95}, 
 ensures  the   existence   of
 $(u,\tau)$, with $\tau\in (0,  +\infty ]$ and $u\in \cc^1([0, \tau))$ taking values in $X$,
 that  is a solution to  the Cauchy
problem:
\begin{equation}\label{eq:gen_ODE}
\left\{\begin{array}{l}
u' = F(u), \\ 
u(0) = x,
\end{array}\right.
\end{equation}
where the first equality in~\eqref{eq:gen_ODE}  holds in $[0, \tau)$ and
$x\in X$ is the so-called initial condition, and furthermore the solution $(u,\tau)$
is  unique and  maximal (that  is, if  $(u',\tau')$ is  another solution
to~\eqref{eq:gen_ODE},   then    $\tau'\leq   \tau$   and    $u'=u$   on
$[0, \tau')$).
We say the solution is \emph{global} 
if  $\tau=+\infty $.

\medskip

We    end   this    section    with   a    comparison   theorem.     Let
$(X,   \norm{\cdot},   \leq)$   be    a   real   Banach   lattice. Let
$D_1,  D_2   \subset  X$.   A  map   $F\,  \colon\,  X  \mapsto   X$  is
\emph{cooperative}    on     $D_1    \times    D_2$    if     for    any
$(x,y) \in D_1 \times D_2$ with $x \leq y$ and any $\nu \in X^\star_+$,
we have:
\begin{equation}
  \label{eq:coop}
\langle \nu, x-y \rangle = 0 \quad\Longrightarrow \quad \langle \nu,
F(x) - F(y) \rangle \leq 0. 
\end{equation}
We recall~\cite[Theorem 2.4]{ddz-sis}.

\begin{prop}[Comparison]\label{prop:comp_th}
  Assume that $X_+$ has non-empty interior. 
Let $F \,\colon\,  X \rightarrow X$ be  locally-Lipschitz, $D_1, D_2 \subset X$ and $\tau \in (0, + \infty]$.
Let $u \, \colon\,  [0, \tau) \rightarrow D_1$ and $v \, \colon\,  [0, \tau) \rightarrow D_2$ be two $\mathcal{C}^1$ maps.
Suppose that $F$ is cooperative on $D_1 \times X$ or on $X \times D_2$,
that  $u(0) \leq v(0)$, and that  $u'(t) - F(u(t)) \leq v'(t) -
F(v(t))$ for all $t\in [0, \tau)$.
Then, we have $u(t) \leq v(t)$ for all $t \in [0, \tau)$.
\end{prop}

 \section{Existence and regularity of the semi-flow for the SIS model}
\label{sec:exists_max_sol}

We prove Proposition~\ref{prop:exists_max_sol}  and
Lemma~\ref{lem:monot_eq_max_para} 
in this section assuming
that  $(T,  \gamma,  \varphi)$  satisfies  Assumption~\ref{assum:1}. 
The
results and proofs are very close to those in \cite{ddz-sis}. 

\subsection{Existence of the semi-flow}
\label{sec:flow}
Recall that Assumption~\ref{assum:1} holds with the function $F$ defined
in~\eqref{eq:def-F}  by $  F(h) = \varphi(h) Th - \gamma h$    for $h
\in L^\infty$. Recall also the set  $\Delta=\{f\in L^\infty \,  \colon\,
\zero\leq f\leq \un\}$. 

\begin{lem}[Regularity of $F$]
  \label{lem:F_loc_lip}
  Let $F$ be the function  from $L^\infty $ to
  $L^\infty $ defined by~\eqref{eq:def-F}. 
 \begin{enumerate}[(i)]
 \item \label{it:Lip-F}
   The function $F$ is  locally-Lipschitz on
     $(L^\infty, \norm{\cdot}_\infty) $.
 \item \label{it:Lip-F-p}
There exists a finite constant $C_p$ such that for $g,h\in \Delta$, we
have:
\[
  \norm{F(g)-F(h)}_{p} \leq  C_p\, 
  \norm{g-h}_{p}.
\]

   \item  \label{it:cont-F}  Let  $(h_n)_{n\in   \N}$  be  a  monotonous
     sequence   of    elements   of   $\Delta$.   Then    the   sequence
     $(F(h_n))_{n\in  \N}$ converges   a.e.\ to  $F(h)$,
     where $h$ is the  a.e.\ limit of $(h_n)_{n\in \N}$.
 \end{enumerate}
\end{lem}

\begin{proof}
Since $\varphi$ is locally Lipschitz, we denote by $K_r$ the
corresponding (finite) Lipschitz constant of $\varphi$ on $[-r, r]$ and
$M_r=\sup_{[-r, r]} |\varphi|\leq  |\varphi(0)| + rK_r$.

We prove Point~\ref{it:Lip-F}. 
Let $r>0$ and $u, v \in L^\infty $ with $\norm{u}_\infty \leq r$ and
$\norm{v}_\infty \leq r$.  We have:
\begin{align*}
  \norm{F(u) - F(v)}_\infty
  & = \norm{\varphi(u)Tu - \varphi(v) Tv - \gamma (u-v)}_{\infty} \\
  & \leq \norm{\varphi(u)}_\infty \norm{T(u-v)}_\infty
    +\norm{\varphi(u) -\varphi(v)}_\infty \, \norm{\ti v}_\infty
    + \norm{\gamma}_\infty \, \norm{u-v}_\infty \\
 & \leq (M_r \norm{T}_{L^\infty} + K_r r \norm{\ti}_{L^\infty}  +
   \norm{\gamma}_\infty ) \norm{u-v}_\infty   .
\end{align*}
This concludes the proof of Point~\ref{it:Lip-F}. 

\medskip

We prove Point~\ref{it:Lip-F-p} in a similar way.
Let $u, v \in \Delta$. We have:
\begin{align*}
  \norm{F(u) - F(v)}_{p}
  & = \norm{\varphi(u)Tu - \varphi(v) Tv - \gamma (u-v)}_{p} \\ 
  & \leq \norm{\varphi(u)}_\infty \norm{T(u-v)}_{p}
    +\norm{\varphi(u) -\varphi(v)}_{p} \, \norm{T v}_\infty
    + \norm{\gamma}_\infty \, \norm{u-v}_{\gamma^p,p} \\
  & \leq \left(M_1 \norm{\tp}_{L^p} + K_1 \norm{T}_{L^\infty}  + 
    \norm{\gamma}_\infty \right) \norm{u-v}_{p}   .
\end{align*}
This concludes the proof of Point~\ref{it:Lip-F-p}.

\medskip

For simplicity,  we assume  that $(h_n)_{n  \in \N}$  is non-decreasing.
Thus it  converges a.e.\ (that  is, in $L^0$) to  a limit, say  $h$, and
this limit belongs to $\Delta$.  Since $T$ is positive, we also get that
the  sequence  $(Th_n)_{n \in  \N}$  is  non-decreasing and  bounded  by
$T\un\in L^\infty  $, thus it converges  a.e.\ (that is, in  $L^0$) to a
limit,  say  $w\in  L^\infty  $.    On  the  other  hand,  by  dominated
convergence, we  also get that  $(h_n)_{n \in  \N}$ converges to  $h$ in
$L^p$, and  thus, as $\tp$ is  bounded on $L^p$,
we   get   that    $(\tp h_n)_{n   \in   \N}$   converges    to   $\tp h$   in
$L^p$. We thus deduce that  $w=\tp h= Th$. Then use that $\varphi$
is continuous,  to deduce that  $(F(h_n))_{n\in \N}$ converges  a.e.\ to
$F(h)$.  This gives Point~\ref{it:cont-F}.
\end{proof}

We now prove that $F$ is cooperative, see~\eqref{eq:coop},  
using that  $\varphi$  is  non-negative  on  $[0,  1]$.

\begin{lem}[$F$ is cooperative]\label{lem:F_coop}
  The map $F$ is cooperative  on $\Delta \times L^\infty$ and on
  $L^\infty \times \Delta$.
\end{lem}

\begin{proof}
We first prove  that $F$ is cooperative on $\Delta \times L^{\infty}$.
Let $u\in \Delta$ and $v\in L^\infty$ with  $u \leq v$.
Let $\nu \in L^{\infty, \star}_+$ such that $\langle \nu, u-v \rangle = 0$. 
Since $v-u \geq 0$, we deduce that
for any $h\in L^\infty $:
\begin{equation}
  \label{eq:null_br}
  \langle
\nu,  (u-v)h \rangle = 0
\end{equation} (see \cite[Lemma 2.6]{ddz-sis} for a
proof in a very similar setting). 
Then, using \eqref{eq:null_br}
with $h=\gamma$, we get: 
\[
  \langle \nu, F(u) - F(v) \rangle
  = \langle \nu, \varphi(u) Tu - \varphi(v) Tv - \gamma (u-v)
  \rangle
  = \langle \nu, \varphi(u) Tu - \varphi(v) Tv 
  \rangle
  .\] 
For $s, t\in \R_+$, we set $\Phi(s,t)=(\varphi(s) -\varphi(t))/(s-t)$ if
$s\neq t$ and $\Phi(s,s)=0$. We have:
\[
  \varphi(u) Tu - \varphi(v) Tv = \varphi(u) T(u-v) +   (u-v) h
  \quad\text{with}\quad
  h=\Phi(u,v)\,  Tv.
\]
As $\varphi$ is locally-Lipschitz by Assumption~\ref{assum:1}
and as $u,v$ and $Tv$ belongs to $L^\infty $, we deduce that $h\in
L^\infty $. We deduce from~\eqref{eq:null_br} that:
\begin{equation}
   \label{eq:Fu-Fv}
  \langle \nu, F(u) - F(v) \rangle = \langle \nu, \varphi(u) T(u-v)
  \rangle .
\end{equation}
As we have $\varphi \geq 0$ on $[0, 1]$ by Assumption~\ref{assum:1} and
$u \leq v$, we have $\varphi(u) T(u-v) \leq 0$. Thus, as $\nu$ is a positive linear form on $L^\infty$, we have $\langle \nu, F(u) - F(v) \rangle \leq 0$.
Therefore the map $F$ is cooperative on $\Delta \times L^{\infty}$.

If $(u,v)\in L^{\infty} \times \Delta$
satisfy $u\leq v$, then using similar computations with $h=\Phi(v,u) T
u$, one get  instead of~\eqref{eq:Fu-Fv} that
 $ \langle \nu, F(u) - F(v) \rangle = \langle \nu, \varphi(v) T(u-v)
 \rangle $. Similar arguments yields then that
  $F$
is also cooperative on $L^{\infty} \times \Delta$.
\end{proof}

Mimicking  the proof  of \cite[Proposition  2.7 (i)]{ddz-sis}  (which in
particular  uses  that 
$\varphi(1)=0$),  we get  that any  solution of~\eqref{eq:intro_gen_SIS}
with an initial condition in $\Delta$ remains in $\Delta$.

\begin{lem}
  \label{lem:D-inv}
The domain $\Delta$ is forward invariant for the differential equation
$u'=F(u)$ in $L^\infty $. 
\end{lem}

We then conclude on the existence of global  solutions in $\Delta$. 

\begin{lem}[Maximal solutions are global]\label{lem:global_sol}
  Any  maximal  solution  of  $u'=F(u)$  in  $L^\infty  $  with  initial
  condition $u(0)=h\in \Delta$ is global.
\end{lem}

\begin{proof}
  The                  bounded                 open                  set
  $U  =  \{f  \in  L^\infty  \, \colon\,  \norm{f}_\infty  <  2  \}$  of
  $L^\infty$ contains $\Delta$,  and the map $F$ is  Lipschitz  on $U$ by
  Lemma~\ref{lem:F_loc_lip}~\ref{it:Lip-F}.   As  the  set  $\Delta$  is
  forward     invariant     by    Lemma~\ref{lem:D-inv},     one     can
  apply~\cite[Corollary IV 1.8]{lang95} to  conclude that any maximal solution
  to $u'=F(u)$ with initial condition in $\Delta$ is global. 
\end{proof}

Under  Assumption~\ref{assum:1},   using  Picard-Lindelöf   theorem~\cite[Corollaries~IV~1.6-8]{lang95}  and
Lemma~\ref{lem:global_sol}, which ensure the existence and uniqueness of
maximal   solution    to   $u'=F(u)$ in $L^\infty $ with initial
condition in $\Delta$,    we   can   define    the   semi-flow
$\phi  \,  \colon\, \R_+  \times  \Delta \rightarrow  \Delta $,  where  the
$L^\infty $-valued  function $\phi(\cdot, h)=( \phi(t,  h))_{t\in \R_+}$
is   the  global   solution  to~\eqref{eq:intro_gen_SIS}   with  initial
condition $u_0=h\in  \Delta$.  Notice  that $\phi(\cdot, h)$  belongs to
$ \cc^1(\R_+)$ and statisfies the semi-group property:
\begin{equation}
   \label{eq:flow-gen}
  \phi(t+s, h)=\phi(t, \phi(s, h))
  \quad\text{for all}\quad
s,t \in \R_+ \quad\text{and}\quad
h\in \Delta.
\end{equation}

\subsection{Properties of the semi-flow}
We now establish the following properties of the semi-flow. 
We stress in the next lemmas that Assumption~\ref{assum:1} holds. 

\begin{lem}[Properties of the semi-flow]\label{prop:mono_sf}
Let $(T, \gamma, \varphi)$ satisfy Assumption~\ref{assum:1}. 

\begin{enumerate}[(i)]
\item  \label{prop:item:order_pre_flow}  If  $h_1 \leq  h_2$  belong  to
  $\Delta$,  then  we  have   $\phi(t,h_1)  \leq  \phi(t,h_2)$  for  all
  $t \in \R_+$.
\item  \label{prop:item:mono_F}  Let  $h \in  \Delta$.   The  function
  $t \mapsto \phi(t,h)$  from $\R^+$ to $L^\infty  $ is non-decreasing
  (resp.  non-increasing) if and only if we have $F(h) \geq \zero $ (resp.
  $F(h) \leq \zero $) in $L^\infty $.
  \item Let $t\in \R_+$. The function $h\mapsto \phi (t, h)$
    defined on $\Delta$ is Lipschitz with respect to
    $\norm{\cdot}_\infty $.
  \item   \label{prop:item:sf_cont}    Let   $t\in   \R_+$.    The   function
    $h  \mapsto \phi  (t, h)$  defined  on $\Delta$  is continuous  with
    respect to the a.e.\ convergence, and, more generally,
    if $(h_r)_{ r\in  \R_+}$ is a sequence of elements  of $\Delta$ such
    that  $h=\limess  _{r\rightarrow +\infty  }  h_r$  exists (and  thus
    belongs               to               $\Delta$),               then
    $\limess _{r\rightarrow +\infty } \phi(t, h_r)=\phi(t, h)$.
\end{enumerate}
\end{lem}

\begin{proof}
  For   all  the   Points   but~\ref{prop:item:sf_cont},   the  proof   mimic
  respectively the  proofs of Propositions  2.8, 2.10 and  2.11~(ii) from
  \cite{ddz-sis}.   Following  the  proof of  Proposition  2.11~(iii)  in
  \cite{ddz-sis},  we see  that to  get Point~\ref{prop:item:sf_cont},  it is
  enough to check the following claim:
   if $(h_n)_{n\in  \N}$ is a monotonous sequence of
  elements of $\Delta$, which thus  converges a.e.\ to a
  limit, say  $h\in \Delta$,  then $(\phi(t, h_n))_{n\in  \N}$ converges
  also a.e.\ to  $\phi(t,h)$ for all $t\geq 0$.

  \medskip
  
  For simplicity, we assume
  that  the   sequence  $(h_n)_{n\in   \N}$  is   non-decreasing.   From
  Point~\ref{prop:item:order_pre_flow},   we  get   that  the   sequence
  $(\phi(s, h_n))_{n\in  \N}$ is also non-decreasing  and thus converges a.e.\  to  a  limit   say  $f_s\in  \Delta$  for  all
  $s\in \R_+$.  We deduce from Lemma~\ref{lem:F_loc_lip}~\ref{it:cont-F}
  that  $(F(\phi(s, h_n)))_{n\in  \N}$  converges a.e.\
  towards $F(f_s)$. Since  $F$ is bounded on $\Delta$ (as  it is locally
  Lipschitz on $L^\infty $) we deduce that the convergence also holds in
  $L^p$.      Since      the     identity      map     from
  $(L^\infty,         \norm{\cdot}_\infty         )         $         to
  $(L^p, \norm{\cdot}_{ p  } )$ is continuous, we
  deduce  that   a  solution  to~\eqref{eq:intro_gen_SIS}  in
  $L^\infty  $   is  also  a  solution   to~\eqref{eq:intro_gen_SIS}  in
  $L^p$.  By  the  fundamental Theorem  of calculus,
  see~\eqref{eq:calculus},   we   get  that   for  all
  $s\geq 0$:
\[
  \phi(s,h_n) = h_n+ \int_0^s F(\phi(r, h_n))\,  \rd r
 \quad\text{and}\quad
  \phi(s,h) = h+ \int_0^s F(\phi(r, h))\,  \rd r
  \quad\text{hold in } L^p.
\]
By the dominated  comvergence Theorem~\ref{th:TCD} (with $X=L^p$ and
$g$ the constant function on $\R_+$ equal to 1), we
deduce that  $(f_s)_{s\in \R_+}$ is Bochner integrable on bounded
intervals of $\R_+$ and  for all $s\geq 0$:
\[
  f_s = h+ \int_0^s F(f_r)\,  \rd r
  \quad\text{holds in } L^p.
\]
We deduce from~\eqref{eq:Bocner-norm} and
Lemma~\ref{lem:F_loc_lip}~\ref{it:Lip-F-p} that for all $s\geq 0$:
\[
  \norm{f_s - \phi(s,h)}_{ p }
  \leq  C_p \int _0 ^s   \norm{f_r - \phi(r,h)}_{ p }\, \rd r. 
\]
Since $f_r$ and $\phi(r,h)$ belong to $\Delta$, we get that $\norm{f_r -
  \phi(r,h)}_{ p }\leq  2$, so that $r\mapsto  \norm{f_r -
  \phi(r,h)}_{ p }$ is locally $\nu$-integrable. 
We deduce from the Grönwall's inequality that $\norm{f_s -
  \phi(s,h)}_{ p }=0$ for all $s\geq 0$. This gives that
$f_s=\phi(s,h)$ for all $s\geq 0$, which proves the claim.
\end{proof}

Following \cite[Proposition  2.13]{ddz-sis}, we prove that  the limit of
the semi-flow is an equilibrium.

\begin{lem}[Limits of the semi-flow are equilibria]\label{lem:lim_eq}
Let $(T, \gamma, \varphi)$ satisfy Assumption~\ref{assum:1}. 
Let $h \in \Delta$. If $\limess_{t \rightarrow +\infty} \phi(t, h)$
exists, then it belongs to $\Delta$ and is an equilibrium. 
\end{lem}

\begin{proof}
  Let $h^*=\limess_{t \rightarrow +\infty} \phi(t, h)$. 
By Lemma~\ref{prop:mono_sf}~\ref{prop:item:sf_cont} and
by~\eqref{eq:flow-gen}, we have for all $s \in \R_+$ that:
\[
  \phi(s, h^*)=\phi(s, \limess_{t\rightarrow +\infty } \phi(t, h))
  = \limess_{t\rightarrow +\infty } \phi(s,  \phi(t, h))=
  \limess_{t\rightarrow +\infty } \phi(t, h)= h^*.
\]
Then use Lemma~\ref{prop:mono_sf}~\ref{prop:item:mono_F} to get that
$F(h^*)=0$. 
\end{proof}

\subsection{Proof of Proposition~\ref{prop:exists_max_sol}  and
Lemma~\ref{lem:monot_eq_max_para}}

\begin{proof}[Proof of Proposition~\ref{prop:exists_max_sol}]
The solution to Equation~\eqref{eq:intro_gen_SIS}  in $L^\infty $ with
initial condition in $\Delta$ is
given by the semi-flow $\phi$, see Section~\ref{sec:flow} and
Lemma~\ref{lem:global_sol} therein. This gives
Point~\ref{prop:item:unique_sol}. Point~\ref{prop:item:delta_inv} is
Lemma~\ref{lem:D-inv}.

\medskip

Since $F(\un) = \varphi(1) T(\un) - \gamma = - \gamma \leq 0$ by
Assumption~\ref{assum:1}, we get  by
Lemma~\ref{prop:mono_sf}~\ref{prop:item:mono_F} that  the semi-flow $t
\mapsto \phi(t, \un)$ is non-increasing. This implies that
$g^*=\limess_{t\rightarrow +\infty } \phi(t, \un)$ exists. By
Lemma~\ref{lem:lim_eq}, we get that $g^*$ is an equilibrium. 
Let $h \in \Delta$ be an equilibrium.
We have $h \leq \un$, thus by
Lemma~\ref{prop:mono_sf}~\ref{prop:item:order_pre_flow} we have $h =
\phi(t,h) \leq \phi(t,\un)$ for all $t \geq 0$. Taking the essential
limit, we get  $h \leq g^*$. 
This  gives Point~\ref{prop:item:g*_lim}.
\end{proof}

\begin{proof}[Proof of Lemma~\ref{lem:monot_eq_max_para}]
  For $h \in L^\infty$ and $i=1,2$, let
  $F_i(h) = \varphi_i(h) T_i(h) - \gamma_i h$ and let $\phi_i$ be the
  semi-flow of Equation~\eqref{eq:intro_gen_SIS} with the parameters
  $(T_i, \gamma_i, \varphi_i)$.  By assumption, we have
  $F_1(g^*_2) \geq F_2(g^*_2) =  0$. Thus, by
  Lemma~\ref{prop:mono_sf}~\ref{prop:item:mono_F}, the semi-flow
  $t \mapsto \phi_1(t, g^*_2)$ is non-decreasing.  By
  Lemma~\ref{lem:lim_eq}, since the  essential limit 
  $g=\limess_{t\rightarrow +\infty } \phi_1(t, g^*_2)$ exists, it belongs to
  $\Delta$ and is an equilibrium (for the parameters
  $(T_1, \gamma_1, \varphi_1)$).  As $g^*_1$
  is the maximal equilibrium, we have $g^*_1 \geq g$, and thus 
  $g^*_1 \geq g^*_2$ as the semi-flow $t \mapsto \phi_1(t, g^*_2)$ is
  non-decreasing.
\end{proof}

\section{Proof of Proposition~\ref{th:DDZ_cv_sf}}\label{sec:DDZ_cv_sf}

\subsection{On the support of the semi-flow}

The following lemma is a generalization of \cite[Lemma 4.10]{ddz-sis}
where  $T$ is assumed to be  irreducible.

\begin{lem}[Support of the semi-flow]\label{lem:sf_pos}
  Let $(T, \gamma, \varphi)$  that satisfies Assumption~\ref{assum:1}
  with $\varphi(0)>0$. Let
  $h \in  \Delta$.  We have $\supp(\phi(t,  h))=\cf(\supp(h))$ a.e.\ for
  all $t > 0$.
\end{lem}

\begin{proof}
  Since $\varphi(0)>0$,  there exists
  $a,\eta\in (0,  1)$ such that  $a-\varphi(r)<0$ for all $r\in  [0, \eta]$.
  Notice  the   operator  $Q=aT  -\gamma+  \norm{\gamma}_\infty   $  on
  $L^\infty $  is positive and that  the invariant sets for  $Q$ and $T$
 are the same.
 Set $f=\eta h/2$  and $A=\supp(h)=\supp(f)$.  There
  exists $c > 0$ small enough such that for all $t \in [0,c]$:
\begin{equation}\label{proof:eq:exp(t||T||)}
\exp({at \norm{T}_{L^\infty} } )< 2.
\end{equation}
Set for  $t \geq 0$:
\begin{equation}\label{proof:eq:def:u(t)}
u(t) = \expp{t(a  T-\gamma)} f= \expp{- \norm{\gamma}_\infty  t} \,
\expp{tQ} f.
\end{equation}
By \cite[Corollary 5.7]{dlz}  applied to the operator $Q$,  we get that
$\supp(u(t))=\supp(u(t)\expp{\norm{\gamma}_\infty \,     t})=\cf(A)$    for    $t>0$.
Differentiating~\eqref{proof:eq:def:u(t)} leads to:
\begin{equation}\label{proof:eq:u'-Fu}
u'(t) - F(u(t)) = (a - \varphi(u(t))) T u(t).
\end{equation}
We deduce from~\eqref{proof:eq:exp(t||T||)}
and~\eqref{proof:eq:def:u(t)} that
$\norm{u(t)}_\infty <\eta$ for  $t\in [0, c]$, and thus, by definition
of $a$ and  $\eta$, that $ u'(t) - F(u(t))\leq 0$ for  $t\in [0, c]$. Then, since
$u(0)=f\leq  h$, 
Theorem~\ref{prop:comp_th} implies that $u(t)\leq \phi(t, h)$ 
for  $t \in [0,c]$, and thus $\cf(A)\subset \supp(\phi(t, h))$ for
$t\in (0, c]$. Then, use the semi-flow equation~\eqref{eq:flow-gen} to
propagate the result to all $t>0$.

We deduce from Lemma~\ref{lem:sf_rest}~\ref{lem:item:rest_inv} that
$\supp(\phi(t,h)) =\supp(\phi_{\cf(A)}(t,h)) \subset \cf(A)$ for all $t\geq
0$. This gives that  
$\supp(\phi(t,h)) = \cf(A)$ for all $t> 0$. 
\end{proof}

\subsection{Preliminary results on the spectral radius and bound}

We refer to \cite[Section 3]{thieme09} for results on the spectral bound
on    Banach     lattices    defined     by~\eqref{eq:def-s(T)}.

\begin{lem}[Spectral radius and spectral bound]
  \label{lem:rad_bound}
  Let $(T, \gamma, \varphi)$  that satisfy Assumption~\ref{assum:2}. Let
  $\delta :  \Omega \rightarrow  \R$ be  a measurable  positive function
  with  $\delta  \geq  \gamma$  and  $\essinf \delta  >  0$.   Then  the
  quantities $s(T - \delta)$, $s(\tp - \delta)$,
  $\rho(\tdi)-1$ and  $\rho(\tg{\delta})-1$   have
  the same sign.
\end{lem}

\begin{proof}
  Notice     that    $(T,     \delta,    \varphi)$     also    satisfies
  Assumption~\ref{assum:2}.  By  Lemma~\ref{lem:T}~\ref{item:spec}  with
  $\gamma$      replaced     by      $\delta$,     we      have     that
  $\rho(\tg{\delta})=\rho(\tdi)$.   So  it  is   enough  to  prove  that
  $s(\tp -  \delta)$ and $\rho(\tg{\delta})-1$  have the same  sign, and
  that $s(T  - \delta)$ and  $\rho(\tdi)-1$ have the same  sign.  This
  can be done by  mimicking the proof of~\cite[Proposition 4.1]{ddz-sis}
  based on~\cite{thieme09}, noticing that the cone $L^p_+$ is normal and
  reproducing for $p\in [1, +\infty  ]$, and that, by Lemma~\ref{lem:T},
  the operators $\tg{\delta}$ and $\tdi$ 
are respectively 
 compact 
on  $L^p$ and   power compact  on $L^\infty $, and that the  linear
maps
$\tp   -  \delta$   and  $T   -  \delta$  are operators
respectively on 
$L^p$ and
$L^\infty $. 
\end{proof}

Adapting the proof of \cite[Proposition 4.2]{ddz-sis} on kernel operators, we provide
a weaker link between $\rho ( \tg{\gamma}) -1$ and $s(\tp -
\gamma)$ without the condition $\essinf \gamma > 0$. 

\begin{prop}[Positive spectral bound and Krein-Rutman theorem]\label{prop:rad_vpd}
Let $(T, \gamma, \varphi)$ that satisfy Assumption~\ref{assum:2}.
Then the following assertions are equivalent:
\begin{enumerate}[(i)]
\item \label{prop:item:bound}
  $s( T - \gamma) > 0$ or equivalently $s(\tp - \gamma) > 0$. 
	\item  \label{prop:item:R_0>1}
          $\rho (\tgi) > 1$ or equivalently  $\rho(\tgp)>1$.
	\item There exists $\lambda > 0$ and $w \in L^{\infty}_+
          \priv{\zero}$ such that we have $Tw - \gamma w = \lambda
          w$. \label{prop:item:exists_vpd} 
\end{enumerate}
\end{prop}

\begin{proof}
  Recall Assumption~\ref{assum:2} holds and $p\in (1, +\infty )$.  Since
  $s(A-(\gamma+\varepsilon))=s(A     -\gamma)      -\varepsilon$     for
  $\varepsilon\in \R$  and $A$  equal to  $T$ or  $\tp$, we  deduce from
  Lemma~\ref{lem:rad_bound}     that    the     two    conditions     in
  Point~\ref{prop:item:bound}  are  equivalent.   We  also  deduce  from
  Lemma~\ref{lem:T}~\ref{item:spec}   that   the   two   conditions   in
  Point~\ref{prop:item:R_0>1} are equivalent. So, we shall only consider
  the second ones.  It is immediate that Point~\ref{prop:item:exists_vpd}
  implies  Point~\ref{prop:item:bound}  as  $L^\infty \subset  L^p$  and
  $T$ and $\tp$ coincide on $L^\infty $.

\medskip

We       assume        Point~\ref{prop:item:bound}       and       prove
Point~\ref{prop:item:R_0>1}.    For   any   $a\in \R_+$,   we   denote
$\psi(a)  =  \rho(V_a)$  with $V_a=\tg{(\gamma+a)}$.   Notice  that
$V_a= \tp M_{1/(\gamma+a)}$ for $a>0$. By
Assumption~\ref{assum:2},  the operator  $V_a$ on $L^p$ is positive
 and that $V_a \geq V_b$ for $0\leq a\leq b$.  Thus the
map $\psi$ is non-increasing on $\R_+$ by~\eqref{eq:spec_rad_croissant}.
By Point~\ref{prop:item:bound}, there exists $\varepsilon > 0$ such that
$s(\tp   -  (\gamma   +  \varepsilon))   =s(\tp   -  \gamma)
-\varepsilon>  0$,   therefore  we  have  $\psi(\varepsilon)   >  1$  by
Lemma~\ref{lem:rad_bound} applied  to $\delta  = \gamma  + \varepsilon$.
We   thus   get   $\psi(0)   =  \rho(\tg{\gamma})   >   1$,   that   is
Point~\ref{prop:item:R_0>1}.

\medskip

We     assume     Point~\ref{prop:item:R_0>1}    and     prove     Point
\ref{prop:item:exists_vpd}.  By   Point~\ref{prop:item:R_0>1},  we  have
$\psi(0)    >    1$.     As    for    all   $a    >    0$,    we    have
$\psi(a) \leq  \norm{V_a}_{L^p} \leq a^{-1}\, \norm{\tp}_{L^p}$. 
  We  deduce  that 
$\lim_{a\rightarrow \infty }\psi(a) =0$.

We now prove that $\psi$ is continuous on $\R_+$.
Let $B$ denote the unit ball
of $L^p$. Notice that $M_{1/(\gamma+a)}(B) \subset M_{1/\gamma}
(B)$ for $a\in \R_+$ and thus $\bigcup_{a\in \R_+}
V_a(B)=V_0(B)=\tg{\gamma} (B)$ is relatively compact in $L^p$, and thus the family
$(V_a)_{a\in \R_+}$ is collectively compact.
Thanks to Lemma~\ref{lem:coll-K},  the continuity of $\psi$ holds if 
$  \lim_{|a-b| \rightarrow 0} \norm{(V_a -V_b)f}_p=0$ 
for any $f\in
L^p$. This is indeed the case as, for $f\in L^p$, we have:
\[
  \norm{(V_a -V_b)f}_p=\norm{ \tg{\gamma} \left(\frac{(b-a)\gamma}{(\gamma+a)
    (\gamma+b)} f\right)}_p
\leq \norm{ \tg{\gamma}}_{L^p}  \norm{ \frac{(b-a)\gamma}{(\gamma+a)
    (\gamma+b)} f}_p, 
\]
and the right members goes to $0$ as $|a-b|$ goes to $0$ 
using  $|b-a|\gamma/(\gamma+a)
(\gamma+b) \leq  1$ and  dominated convergence. In conclusion, the
function  $\psi$ is continuous on $\R_+$. 

Since
$\psi(0) > 1$ and 
$\lim_{a\rightarrow \infty }\psi(a) =0$, we deduce from  the continuity of $\psi$, 
that 
there exists  $\lambda > 0$ such that $\psi(\lambda) = 1$.  Thus by
the Krein-Rutman Theorem~\ref{theo:rappel}~\ref{thm:KR}  applied to the positive compact
operator $V_\lambda $ on $L^p$, there exists
$v \in L^p_+ \priv{\zero}$ such that 
$V_\lambda v = v$. Thanks to~\eqref{eq:T/gamma-in-Lp}, we have 
$\norm{v}_\infty =\norm{V_\lambda v}_\infty\leq
\norm{T(v/\gamma)}_\infty \leq C'_p \, \norm{v}_p$, 
we deduce that $v$
belongs also to $L^\infty $. 
Setting 
$w = v/(\gamma + \lambda)\in L^\infty _+ \priv{\zero}$, we get that 
$\tp w - \gamma w = \lambda w$.  As $\tp$ and $T$ coincide on $L^\infty $, we get 
Point~\ref{prop:item:exists_vpd}.
\end{proof}

\subsection{Proof of Proposition~\ref{th:DDZ_cv_sf}~\ref{th:item:R0<1}}
Let $g$ be  a non-zero equilibrium. 
By
Assumption~\ref{assum:2}, $\varphi(g) < 1$ on $\supp(g)$ (as $\varphi(r)
<1$ for $r\in (0, 1]$).
Recall
 the operator $S=\tg{\gamma}$ is compact on $L^p$,
 see Lemma~\ref{lem:T}~\ref{item:tg}. Since $g\in \Delta$ is an
 equilibrium, we obtain that:
\[
S (\gamma  g)  =  \frac{\gamma   g}{\varphi(g)}> \gamma g
  \quad\text{on}\quad \supp(g).
\]
We deduce from 
Lemma~\ref{lem:rad>1}~\ref{lem:item:superstrict}, with $\lambda=1$, 
and~\eqref{eq:spec_rad_croissant} that
$R_0 =\rho(S)\geq    \rho   \left(S_{\supp(g)}   \right)
>1$.   In other  words, $R_0\leq  1$ implies
that $\zero$ is the only equilibrium. The last part of Point~\ref{th:item:R0<1}
is a consequence of
Proposition~\ref{prop:exists_max_sol}~\ref{prop:item:g*_lim} and the
monotonicity of the semi-flow from
Lemma~\ref{prop:mono_sf}~\ref{prop:item:order_pre_flow}. 

\subsection{Proof of Proposition~\ref{th:DDZ_cv_sf}~\ref{th:item:R0>1}}
We  assume  that  we  have   $R_0  >  1$.   Similarly  to~\cite[Section
4.4]{ddz-sis},  we  will prove  that  there  exists a  non-zero  initial
condition  $w  \in \Delta$  such  that  the  semi-flow $\phi(.,  w)$  is
non-decreasing.  As $R_0 > 1$,  there exists $a \in (0, 1)$ such
that  $\rho \left( a\tg{\gamma} \right) >  1$.  Thus, by
Proposition~\ref{prop:rad_vpd} (with  $ (a T, \gamma,  \varphi)$), there
exists $\lambda >  0$ and $w \in L^\infty_+ \priv{\zero}$  such that
$a Tw -  \gamma w = \lambda w$.  By  Assumption~\ref{assum:2}, the map
$\varphi$  is  continuous  with  $\varphi(0) =  1$;  thus  there  exists
$\eta  \in  (0,1)$  such  that  for   all  $r  \in  [0,\eta]$,  we  have
$\varphi(r)  \geq  a$.   Without  loss of  generality,  we  assume  that
$\norm{w}_\infty \leq \eta$, and thus $w \in \Delta$.  We deduce that:
\[
  F(w) = \varphi(w) Tw - \gamma  w \geq aTw - \gamma w
  =          \lambda          w           \geq          0.
\]
By   Lemma~\ref{prop:mono_sf}~\ref{prop:item:mono_F},    the   semi-flow
$t  \mapsto  \phi(t,  w)$  is  thus non-decreasing  on  $\R_+$  and  its
essential   limit,  say   $g$,   exists and belongs to $\Delta$.  It   is   an  equilibrium   by
Lemma~\ref{lem:lim_eq}. Let $g^*$ denote  the maximal equilibrium. As we
have $g^*\geq  g \geq w$  and $\mu(\supp(w)) >  0$, we   deduce that
$\mu(\supp(g^*)) > 0$.

\subsection{Proof of Proposition~\ref{th:DDZ_cv_sf}~\ref{th:item:R0>1-irr}}

We now assume  that $T = T_A$  with $A$ an irreducible  set.  Notice the
set $A$ is invariant and thus admissible; and it has positive measure as
$R_0>0$.  It is  thus a  (non-zero) atom  by~\cite[Theorem~1]{dlz}.  Let
$g$  be  a non-zero  equilibrium.    We  have  $\supp(g)  \subset  A$
by~\eqref{eq:gen_equilibre}.    Since  $\supp(g)$   is  invariant,   see
Lemma~\ref{lem:eq_eigenf}~\ref{lem:item:supp_inv}, and  $A$ is  an atom,
we  deduce  that  $\supp(g)=A$.  Then   use  that  the  support  of  $g$
characterizes $g$, see Corollary~\ref{cor:supp-equi=} to deduce that $g$
is the only non-zero equilibrium.

 \subsection{Proof of Proposition~\ref{th:DDZ_cv_sf}~\ref{th:item:R0>1-irr-lim}}
 By Point~\ref{th:item:R0>1-irr}, we have $\supp(g^*) = A$.
 On   $A^c$,  we   have   $\phi(t,h)'=  -\gamma   \phi(t,h)$,  so   that
 $\limess_{t\rightarrow \infty } \phi(t,  h) \ind{A^c}=\zero$.  So it is
 enough to prove  the result when $\supp(h)\subset  \supp(g^*) = A$. As $\supp(h)$ is a non-empty set included in the invariant and irreducible set $A$, its future is equal to $A$. Thus, by
 Lemma~\ref{lem:sf_pos} and considering $\phi(1, h)$ instead of $h$, one
 can assume without loss of generality that $\supp(h) = \supp(g^*) = A$.

 For    $\varepsilon\in   (0,    1]$,   we    consider   the    operator
 $U_\varepsilon=\varphi(\varepsilon)      M_{\{h\geq      \varepsilon\}}
 \tg{\gamma}$     on     $L^p(\mu)$,    and     set     $U_0=M_{\{h>0\}}
 \tg{\gamma}$.    Let $B$ be  the unit ball in
 $L^p$.       Since      $\tg{\gamma}$     is      compact      and
 $\lim_{\varepsilon  \rightarrow 0}\varphi(  \varepsilon) \ind{\{h  \geq
   \varepsilon\}}   =    \ind{\{h>   0\}}$   a.e.,   we    deduce   that
 $\bigcup  _{\varepsilon\in  [0,  1]}  U_\varepsilon(B)$  is  relatively
 compact.        Thus        the       family        of        operators
 $(U_\varepsilon)_{\varepsilon\in [0,1]}$ is collectively compact.
 By dominated convergence, we also get that
$\lim_{\varepsilon\rightarrow 0} \norm{(U_\varepsilon -U_0) f}_p=0$. We
deduce from Lemma~\ref{lem:coll-K} that the map $\varepsilon \mapsto \rho(U_\varepsilon)$ is
continuous at $0$. Thus, there exists $\varepsilon \in (0,1)$ such that  $\rho (U_\varepsilon) > 1$.
Notice that $( \varphi(\varepsilon) \ind{\{h \geq
   \varepsilon\} }T, \gamma, \varphi)$ satisfies Assumption~\ref{assum:2}.
   By Proposition~\ref{prop:rad_vpd} with $T$ replaced by
   $\varphi(\varepsilon)  \ind{\{h \geq
   \varepsilon\}} T$,  there exists $\lambda > 0$ and $w \in
   L^\infty_+ \priv{\zero}$ such that:
\begin{equation}
  \label{eq:wA}
  \varphi(\varepsilon)\ind{\{h \geq \varepsilon\}} T w - \gamma w= \lambda w.
\end{equation}
Without loss of generality, we can assume that $\norm{w}_\infty \leq
\varepsilon$. We also have:
\begin{equation}
  \label{eq:wA=0}
  \supp(w) \subset \{h\geq  \varepsilon\} \subset A.
\end{equation}
Using~\eqref{eq:wA}
we get that:
\begin{equation}
   \label{eq:F(w)>0}
F(w)=\varphi(w) Tw - \gamma w\geq  \varphi(\varepsilon)  Tw - \gamma w
\geq \lambda w\geq 0.
 \end{equation}
 We deduce that the map $t \mapsto \phi(t,w)$ is non-decreasing by
 Lemma~\ref{prop:mono_sf}~\ref{prop:item:mono_F} and, as $g^*$
 is the only non-zero equilibrium,  that $\limess_{t\rightarrow \infty }
 \phi(t, w)=g^*$ by Lemma~\ref{lem:lim_eq}. 
 As the semi-flow is monotone by
 Lemma~\ref{prop:mono_sf}~\ref{prop:item:order_pre_flow}, we deduce that
 $\phi(t, w) \leq 
\phi(t, h)\leq  \phi(t,\un)$ for all $t\in \R_+$. Then use
Proposition~\ref{prop:exists_max_sol}~\ref{prop:item:g*_lim} to conclude
that  $\limess_{t\rightarrow \infty }
 \phi(t, h)=g^*$.

\section{Proof of Theorem~\ref{th:cv_eq_max}}\label{sec:sf_cv}

Let $(T, \gamma, \varphi)$ that satisfy Assumption~\ref{assum:2}.
We keep notations from Section~\ref{subsec:carac_eq}; so 
$\cc_A$ is   the
supercritical antichain given by the maximal elements of the
supercritical atoms included in the  set $A$.

\begin{proof}[Proof of Theorem~\ref{th:cv_eq_max}]
  Let $A = \cf(\supp(h))$ and $g^*_A$  be the maximal equilibrium on $A$
  (notice  that $A=\emptyset$  if  $h=\zero$). Since  $A$ is  invariant,
  $g^*_A$        is         also        an         equilibrium        by
  Lemma~\ref{cor:inv-equi}~\ref{it:equiA-equi}.
  
We will, as in the proof of Proposition~\ref{th:DDZ_cv_sf}~\ref{th:item:R0>1-irr-lim}, prove the existence of $w \in \Delta$ with $w \leq h$ such that the semi-flow $(\phi(t, w))_{t\in \R_+}$ is non-decreasing and converges essentially to $g^*_A$.
By Lemma~\ref{lem:sf_pos}  and considering $\phi(1, h)$  instead of $h$,
one can  assume without loss of  generality that $\supp(h) =A$.
If $\cc_A$ is empty, we get that $g^*_A=\zero$ and 
by Lemma~\ref{lem:sf_rest}~\ref{lem:item:CV_sf_unA}
and~\ref{lem:item:rest_inv}, we have $ \limess_{t\rightarrow+\infty }
\phi(t, \un_A) = \zero$, and by monotonicity of the semi-flow that
$\limess _{t\rightarrow+\infty }
\phi(t, h) =\zero$, which proves Theorem~\ref{th:cv_eq_max} in this
case.

\medskip

We now assume that $\cc_A$ is not empty. 
Let $B\in \cc_A$. By considering $T_B$ instead of $T$, mimicking 
the proof of Proposition~\ref{th:DDZ_cv_sf}~\ref{th:item:R0>1-irr-lim},
see~\eqref{eq:wA=0} and~\eqref{eq:F(w)>0}, we deduce that 
there exists a function $w_B\in \Delta$ such that $\supp(w_B)\subset B$,
$w_B\leq  h$, $F(w_B)\geq 0$ and $\limess_{t\rightarrow \infty }
\phi_B(t,w_B)= g^*_B$. 
Since $\cc_A$ is an antichain of atoms, we deduce that $\cf(\supp(w_B))\cap 
\supp(w_{B'})\subset \cf(B) \cap B'=\emptyset$ for all $B', B\in \cc_A$ such that $B\neq B'$. Set
$w=\sum_{B\in \cc_A} w_B\leq h$. We have:
\[
  F(w)
  =\sum_{B\in \cc_A} F(w_B) + \ind{\cf(B)} \left( \varphi(w)- \varphi(w_B)\right) T w_B
  =\sum_{B\in \cc_A} F(w_B)\geq 0,
\]
where for  the first  equality we  used that  $\cf(B)$ is  invariant and
$\supp(w_B)  \subset   B\subset  \cf(B)$,   and  for  the   second  that
$\varphi(w)-\varphi(w_B)=0$ on $\cf(B)$ for $B\in \cc_A$.  Arguing as in
the end of the proof  of Proposition~\ref{th:DDZ_cv_sf}~\ref{th:item:R0>1-irr-lim},
we deduce that the semi-flow $(\phi(t,w))_{t\in \R_+}$ is non-decreasing
and that  $\limess_{t\rightarrow \infty  } \phi(t, w)=g\in  \Delta$ with
$g$  an equilibrium.

\medskip

We now prove the equality $\cc_A = \cc_g = \cc_{g^*_A}$.
Since $w\leq  g$,
we deduce that $\cc_A \subset \cc_g$. 
Since $\supp(w) \subset A$ and $A$ is invariant,  we have   $\supp(g)
\subset A$ by Lemma~\ref{lem:sf_pos}. In particular, $g$ is an
equilibrium of $A$ by Lemma~\ref{cor:inv-equi}~\ref{it:equi-equiA}, thus
we get $g \leq g^*_A$ and deduce that  $\cc_g \subset \cc_{g^*_A}\subset
\cc_A$ as $\supp(g^*_A) \subset A$.
This proves the claim and that  $g=g^*_A$ by Corollary~\ref{cor:supp-equi=}. 

By monotonicity of the semi-flow, we have that $ \phi(t,w) \leq
\phi(t,h) \leq \phi(t, \un_A)$ for all $t \geq 0$.
The left term converges essentially to $g=g^*_A$, and the  right term to
$g^*_A$ by Lemma~\ref{lem:sf_rest}. This gives
 $\limess_{t\rightarrow \infty  } \phi(t,h)=g^*_A$. 
This ends the proof. 
\end{proof}

\end{document}